\documentclass[12pt]{amsart}

\usepackage{dsfont}
\usepackage{amsmath}
\usepackage{amssymb}
\usepackage{graphicx}

\usepackage{amssymb}

\usepackage{amsfonts}

\usepackage{soul}

\usepackage{mathpazo}
\usepackage{color}
\usepackage{yfonts}

\usepackage{paralist}

\usepackage{stmaryrd}

\usepackage{amsxtra}
\def\o{\omega}

\def\Id{{\rm Id}}

\def\eps{\varepsilon}

\def\cI{\mathcal I}
\newcommand{\cR}{\mathcal R}
\newcommand{\cL}{\mathcal L}

\newcommand{\norm}[1]{\| #1\|}

\def\a{          \alpha}

\def\cJ{          \mathcal J}

\def\cH{          \mathcal H}

\def \R{{\mathbb R}}

\def \Z{{\mathbb Z}}
\def \N{{\mathbb N}}

\def \l{{\lambda}}

\newcommand{\AAA}{{\mathbb A}}

\newcommand{\T}{{\mathbb T}}
\newcommand{\prf}{{\begin{proof}}}
\newcommand{\epf}{{\end{proof}}}

\newcommand{\Q}{{\mathbb Q}}
\newcommand{\DD}{{\mathbb D}}

\newtheorem{theo}{Theorem}
\newtheorem{lemma}{Lemma}[section]
\newtheorem{sublemma}{Sublemma}[section]
\newtheorem{cor}{Corollary}[section]
\newtheorem{prop}{Proposition}[section]

\theoremstyle{definition}
\newtheorem{definition}{Definition}[section]

\newtheorem{rema}{Remark}[section]

\newtheorem{question}{Question}

\newtheorem{Main}{Theorem}

\def\bee{\begin{equation}}
\def\eee{\end{equation}}

\newcommand{\pdvr}[2]
{\dfrac{\partial^{#2} #1}{\partial \theta^{#2_1} \partial r^{#2_2}}}
\newcommand{\pdvrs}[2]
{\partial^{#2} #1 /\partial \theta^{#2_1} \partial r^{#2_2}}
\newcommand{\ZZ}{{\mathbb Z}}

\numberwithin{equation}{section}

\author{A. Avila, B. Fayad, P. Le Calvez, D. Xu, Z. Zhang}
\begin{document}

\title[On mixing  diffeomorphisms of the disk]{On mixing  diffeomorphisms of the disk}



\date{\today}




\maketitle

\begin{abstract} We prove that a real analytic pseudo-rotation $f$ of the disc or the sphere is never topologically mixing. When the rotation number of $f$ is of Brjuno type, the latter follows from a KAM theorem of R\"ussmann on the stability of real analytic elliptic fixed points. In the non-Brjuno case, we prove that a pseudo-rotation of class $C^k$, $k\geq 2$, is $C^{k-1}$-rigid using the simple observation, derived  from Franks' Lemma on free discs, that a pseudo-rotation with small rotation number compared to its $C^1$ (or H\"older) norm must be close to Identity.  

From our result and a structure theorem by Franks and Handel (on zero entropy surface diffeomorphisms) it follows that an analytic  conservative diffeomorphism of the disc or the sphere that is topologically mixing must have positive topological entropy.

In our proof we need an {\it a priori} limit on the growth of the derivatives of the iterates of a pseudo-rotation that we obtain {\it via} an 
 effective finite information version of the Katok closing lemma for an area preserving surface diffeomorphism $f$, that provides a controlled  gap in the possible growth of  the derivatives of $f$ between exponential and sub-exponential.

\end{abstract}
\section{Introduction}

It is a known fact that smooth dynamical systems may display very wild dynamical behavior. The problem of differentiable realization, that consists of finding  smooth dynamical systems  on compact manifolds that are isomorphic to a given arbitrary abstract measurable dynamical system (with finite entropy),  is wide open. In the one dimensional setting, that is on the circle, smooth realizations are clearly limited since a circle diffeomorphism that preserves a smooth measure is smoothly conjugated to a rotation. Apart from that, the only case where an obstruction to smooth realization is known to exist is in surface dynamics, where Pesin theory asserts that an ergodic area preserving diffeomorphism of class $C^2$ that has a positive metric entropy must be isomorphic to a Bernoulli shift \cite{pesin}. For zero entropy systems, no obstruction to smooth realization is known to exist apart from the one dimensional case. 

In the setting of area preserving (or conservative) two dimensional disc or sphere diffeomorphisms, that will be our interest here, some analogies with circle diffeomorphisms push in the direction of finding obstructions to smooth realization, even in the zero entropy case. Among these analogies, is the fact that  a rotation set can be defined for an area preserving disc or sphere diffeomorphism $f$ and if this set contains rationals then $f$ must have periodic points \cite{franks}. Another one is that if the rotation number of $f$ at some elliptic fixed point or on the boundary of the disc is Diophantine, then the smooth linearization phenomenon of circle diffeomorphisms with Diophantine rotation number extends to a neighborhood of the fixed point or the  boundary \cite{russmann,FK}. We will see in the sequel how these analogies influence the smooth realization problem. 

On the other hand, examples of smooth dynamical systems with a rich variation of ergodic properties were constructed on the disc and the sphere. Katok constructed a smooth conservative map of the disc that is isomorphic to a Bernoulli shift \cite{katok}. His example, that is based on slowing down fixed points of an Anosov automorphism of the two torus before projecting down the dynamics on the sphere, can be made real analytic if the slowing down is chosen adequately.  

On the zero entropy side, the Anosov-Katok construction method \cite{AK} {\it via} successive conjugations of rational rotations 
provided a multitude of constructions with various ergodic properties, starting with the first examples of smooth ergodic conservative diffeomorphisms on the disc or the sphere. The method also yields weak mixing systems, but all the smooth examples on the disc or the sphere obtained up to now were {\it $C^\infty$-rigid} by construction.  A diffeomorphism of class $C^k$, $k\in \N \cup \{\infty\}$, is said to be {\it $C^k$-rigid} if there exists a sequence $q_n$ such that $f^{q_n}$ converges to the Identity map in the $C^k$ topology. If we only know that the latter holds in a fixed neighborhood of some point $p$, we say that $f$ is $C^k$ locally rigid at $p$. Obviously, rigidity or local rigidity precludes mixing. Hence, the following natural question was raised in \cite{FaKa} in connection with the smooth realization problem and the Anosov-Katok construction method.

\begin{question} \label{q.FK} Does there exist a smooth area preserving diffeomorphism of the disc that is mixing with zero metric entropy? with zero topological entropy?
\end{question}

We will show here that 
\begin{Main} \label{zeroentropy} Area preserving real analytic diffeomorphisms of the disc or the sphere that are topologically mixing have positive topological entropy.
\end{Main}

A homeomorphism $f$ of a metric space $X$ is said to be topologically mixing if for any non empty open sets $A$ and $B$ there exists $N\geq 0$ such that for any $n \geq N$, $f^n(A) \cap B \neq \emptyset$. 

Observe that the existence of mixing homeomorphisms \- of the sphere  (or the disc) with zero metric entropy follows from Katok's construction of an area preserving ergodic diffeomorphism $f$ of the sphere with positive entropy \cite{katok}. Indeed, $f$ as constructed in \cite{katok} is isomorphic to a Bernoulli shift $T$, and it is possible to construct using this isomorphism a mixing invariant measure for $T$ with zero entropy that gives a positive mass to every cylinder.  The almost product structure of stable and unstable leaves of $f$ then imply that the constructed measure assigns a positive mass to every open set in the sphere.
Applying Oxtoby and Ulam's lemma, we get a $C^0$ conjugacy to a mixing conservative homeomorphism of ${\mathbb S}^2$ with zero metric entropy.  

We recall that on the other end of the regularity spectrum, the real analytic category, nothing is known, apart from Theorem \ref{zeroentropy}, about the possible dynamical properties of disc and sphere area preserving zero entropy diffeomorphisms. The following natural question is wide open (see \cite{FaKa,FaKa-analytic} and discussion therein).

\begin{question} \label{qqq2} Does there exist real analytic area preserving diffeomorphisms of the disc or the sphere with zero metric entropy that are transitive? ergodic?
\end{question}

\section{Pseudo-rotations}

\subsection{Analytic pseudo-rotations are locally rigid.} We denote the unit disc in $\R^2$ by 
$\DD= \{ (x,y)\in \R^2 : x^2+y^2\leq 1\}$. A homeomorphism of $\DD$ that preserves Lebesgue measure, fixes the origin,  and has no other periodic points is called a {\it pseudo-rotation}. 
The fixed point at the origin is then necessarily elliptic with irrational rotation number and this rotation number $modulo \ 1$ is denoted by $\rho(f)$ since every point of the disc must have the same rotation number around the origin (see for example \cite[Corollary 2.6]{franks.etds} or \cite[Theorem 3.3]{franks}). 

 All constructions of transitive smooth area preserving diffeomorphisms on $\DD$ by the Anosov-Katok method are pseudo-rotations with a Liouville rotation number. That the rotation number must be Liouville follows from Herman's last geometric Theorem asserting that if this is not the case, then the center and the boundary are accumulated by invariant closed curves which clearly excludes transitivity \cite{FK}.

On the other hand, the structure Theorem by Franks and Handel  for area preserving diffeomorphisms of genus zero surfaces with zero topological entropy shows that transitivity in this case implies that the diffeomorphism must be a pseudo-rotation \cite{franks-handel}. Indeed, they show that if $F$ is an area preserving diffeomorphism of the sphere with at least three fixed points and if $M={\mathbb S}^2-\text{Fix}(F)$ then an open and dense set  $\mathcal W \subset M$  (of points with a weak type of recurrence called {\it weakly free disc recurrent points}) decomposes into a countable union of disjoint invariant annuli on which the dynamics of $F$ is similar to an integrable map albeit with possibly complicated sets (such as pseudo-circles)  for the constant rotation number sets instead of nice simple closed curves. This decomposition clearly excludes transitivity of $F$, hence transitive  area preserving diffeomorphisms on $\DD$  must belong to the class of pseudo-rotations, and to prove Theorem \ref{zeroentropy}, we are reduced to the pseudo-rotation case.

In the case of real analytic pseudo-rotations we indeed show more than the absence of mixing. 

\begin{Main} \label{mixing} Real analytic pseudo-rotations of $\DD$ are $C^\infty$ locally rigid around the center. 
\end{Main}

Bramham already showed in \cite{bramham} that smooth pseudo-rotations with super-Liouville rotation number are $C^0$-rigid.  Our proof of Theorem \ref{mixing}, gives a very simple proof of the latter result based on Frank's free disc Lemma related to Brower's theory of fixed point free planar homeomorphisms (see Theorem \ref{Brjuno} and comments thereafter).  
To cover pseudo-rotations with non-Brjuno type rotation numbers,   we need an additional {\it a priori} control on the growth of the derivatives of the iterates of a pseudo-rotation that we obtain {\it via} an 
 effective finite information version of the Katok closing lemma for area preserving surface diffeomorphisms.

$\bullet$  A number $\a \in \R/\Z-\Q$ is said to be of Brjuno type if 
$$\sum_{n=0}^\infty  \frac{\ln (q_{n+1})}{q_n} < +\infty$$
 where $q_n$ is the sequence of denominators of the best rational approximations to $\a$.

$\bullet$ A numbers $\a \in \R/\Z-\Q$ is said to be {\it super-Liouville} if  
$$\limsup q_n^{-1} \ln q_{n+1} = +\infty$$

To get $C^0$-rigidity, we only require Lipschitz regularity of the pseudo-rotation in the case of super-Liouville rotation number and $C^2$ in the non-Brjuno case. As we will see in the sequel, our direct approach to rigidity of pseudo-rotation cannot go much beyond non-Brjuno type numbers (for example to include all Liouville numbers).  To bridge up with the numbers that are of Brjuno type, we use the KAM (Kolmogorov Arnol'd Moser) result of R\"ussmann that asserts that an elliptic fixed point, that is precisely of Brjuno type,  of an analytic area preserving surface diffeomorphism is surrounded by invariant curves \cite{russmann}, which in the case of pseudo-rotations yields local rigidity. This is where the real analytic hypothesis is necessary in Theorem \ref{mixing}. 

{It is worth mentioning here (see also Question \ref{qqq2}) that no examples of real analytic pseudo-rotations are yet available. Of course, if Birkhoff's conjecture that such examples do not exist turns out to be true, then the content of Theorem \ref{mixing} is void. This however, would not affect our study of rigidity of pseudo-rotations with well approximated rotation numbers since our results hold in low regularity. We now proceed to the precise statements in this context.}

\subsection{Rigidity in the non-Brjuno case.}    Our main rigidity result for pseudo-rotations is the following. 

\begin{theo}  \label{Brjuno} If $f$ is a $C^k$, $k\geq 2$, pseudo-rotation of $\DD$ with rotation number $\a=\rho(f)$ that is not of Brjuno type, then $f$ is $C^{k-1}$-rigid :  There exists a subsequence $q_{n_j}$ of the sequence $q_n(\a)$ such that 
$f^{q_{n_j}} \to {\rm Id}_{\DD}$ in the $C^{k-1}$ topology. 
\end{theo}

When the rotation number is well approximated by rationals we do not need too much regularity to show rigidity of a pseudo-rotation. Indeed, for pseudo-rotations with super-Liouville rotation numbers we have the following rigidity result in which the case H\"older with exponent $a=1$ corresponds to Lipschitz maps. 

\begin{theo} \label{superliouville} If $f$ is a pseudo-rotation of  the disc $\DD$ that is H\"older with exponent  $a \in (0,1]$ and if  $\a=\rho(f)$ satisfies 
$$\limsup  q_n^{-1} a^{q_n} \ln q_{n+1} = +\infty $$
then $f^{q_{n_j}} \to {\rm Id}_{\DD}$ in the uniform topology for any sequence $n_j$ such that $\lim {q_{n_j}}^{-1}  a^{q_{n_j}} \ln q_{n_j+1} = +\infty$.
\end{theo}

For any given modulus of continuity, we can by the same method obtain arithmetic criteria that imply rigidity of a pseudo-rotation with this modulus of continuity. 

\medskip

Theorems \ref{Brjuno} and \ref{superliouville} generalize an earlier result by Bramham \cite{bramham}. Our proof is much more elementary than Bramham's proof that relies on the theory of pseudo-holomorphic curves. The essential idea in our proof is that if a pseudo-rotation $f$ has a rotation number  that is small compared to $\|D f\|^{-2}$ then $f$ must be close to the Identity, because otherwise $f$ would have a periodic point outside the origin which is a contradiction. The proof of the latter fact is based on  a simple observation that combines Franks' free disc lemma with Kac's lemma to give a lower bound on the measure of 
free topological discs of a pseudo-rotation with a small rotation number (see Lemma \ref{main-rigid} and Corollary \ref{cor.displacement} in Section \ref{sec.rigidity}). 
The rigidity for super-Liouville pseudo-rotations then follows immediately since $\|D f^{q}\|$ is at most exponential in $q$. To obtain the rigidity of non-Brjuno numbers, we need a better bound than exponential on the growth of  $\|D f^{q_n}\|$. That such a bound exists {\it a priori} for pseudo-rotations, and more generally for any surface area  preserving $C^2$ diffeomorphism of a surface  without hyperbolic periodic points, is provided by an {\it effective finite information version of the Katok closing lemma for area preserving surface diffeomorphisms} with some hyperbolic behavior. The following result that gives a quantitative gap in the growth of $\|Df^m\|$ between exponential and sub-exponential for an area preserving disc diffeomorphism of class $C^2$  is sufficient, and almost necessary, to prove rigidity for non-Brjuno rotation numbers. 

\begin{Main} 
\label{1.corollary 1.1}
For any compact subset $K \subset {\rm Diff}^{2}_{\text{vol}}(S)$, there exists $0<\theta<1$, integer $H>0$ satisfying the following property:
  Let $\{q_n\}$ be a sequence such that $q_0\geq H$, $q_n \geq H^{q_{n-1}}$, if for some $f\in K$ there exists  $n\geq 0$ such that
\begin{eqnarray}
\frac{1}{q_n}\log\norm{Df^{q_n}}>\theta^n,  \label{cr1.1}
\end{eqnarray}
then $f$ has a hyperbolic periodic point.

\end{Main}

Observe that our result implies that in case \eqref{cr1.1} holds for some $n$, then actually 
$\limsup_{n}  \frac{1}{q_n}\log\norm{Df^{q_n}}>0$, hence  establishing a gap in the possible rate of growth of $\|Df^n\|$.  In spirit, our approach gives an effective instance of what  Katok described as Anosov and Bowen's observation "that assuming some hyperbolicity conditions, dynamical phenomena which  
are observed to almost occur for some diffeomorphisms usually do occur for that diffeomorphism" \cite[Introduction]{katok}.

{ Note that a finite information version of the Katok closing lemma was given by  Climenhaga and Pesin in \cite{CP} that relies on {\it the notion of effective hyperbolicity}. A piece of orbit being called essentially hyperbolic if along this piece there are nicely behaving {\it stable} and {\it unstable} directions for the differential in the sense that expansion along the {\it unstable} direction  dominates the defect from domination over the {\it contracted} direction and dominates also the decay of the angles between the directions under iteration.

Since we want a finitary version of the closing lemma that only involves the growth of derivatives we prefer to give an independent statement from \cite{CP}, with a proof that takes some ideas from \cite{CP}, but differs in that the condition of  conservation of area is used to control {\it a posteriori} the angles as soon as contraction and expansion are controlled.  In the general setting of \cite{CP}, it is not assumed 
 that the maps are area preserving surface diffeomorphisms. Ê Also, our proof does not use graph transforms but rather follows the orbit of a small box along a selected piece of orbit (a kind of effectively hyperbolic piece of orbit) and shows that the return of the box to its neighborhood enjoys hyperbolic-like  properties that force the existence of a hyperbolic periodic point (see Section \ref{katok}).

}
We will give the proofs of Theorems  \ref{Brjuno} and \ref{superliouville} in Section \ref{sec.rigidity}. Theorem \ref{1.corollary 1.1}, that is of independent interest  and of independent flavor from the rest of the paper is discussed and proven in Section \ref{katok}.

Note that Theorem \ref{mixing} has an interesting consequence on the $C^\infty$ centralizer of an analytic  pseudo-rotation $f$ (smooth diffeomorphisms that commute with $f$), since it immediately implies that the latter is always uncountable. Similarly, Theorem \ref{Brjuno} shows that the $C^{k-1}$ centralizer of a $C^k$  pseudo-rotation $f$  with rotation number that is not of Brjuno type is uncountable. This shows that real analytic conservative diffeomorphisms of the disc, or smooth non-Brjuno pseudo rotations are in this extent less flexible than smooth circle diffeomorphisms for which Yoccoz constructed smooth examples with a $C^\infty$ centralizer that is reduced to the powers of the diffeomorphism \cite{yoccoz}. We observe that it is not known if there exists real analytic examples of the latter construction of Yoccoz. 

\subsection{Local rigidity around the center in the Brjuno type case.}  
To go from Theorem \ref{Brjuno} to Theorem \ref{mixing} we use a  KAM  result of R\"ussmann, that does not require any twist condition, and allows to deal with the pseudo-rotations having Brjuno type rotation numbers.

\begin{theo}[R\"ussmann] \label{russmann}  If $f$ is an area preserving real analytic map defined in some open neighborhood of $0 \in \R^2$, and if $f(0)=0$ and the rotation number $\a$ of $f$ at $0$ is of Brjuno type, then $0$ is surrounded by a positive measure set of real analytic invariant closed simple curves. If $f$ has no periodic points accumulating $0$ then $f$ is analytically conjugated to the disc rotation $R_\a$ in some neighborhood of $0$.  
\end{theo}

Theorem \ref{mixing} is then an immediate  consequence of R\"ussmann's Theorem  (in the Brjuno type case) and our rigidity result of Theorem \ref{Brjuno} (in the non-Brjuno type case).

Actually, if in R\"ussmann's Theorem $f$ is in addition supposed to be a pseudo-rotation close to the rotation $R_\a$, it holds that $f$ is analytically conjugated to $R_\a$ and {\it a fortiori} $C^\o$-rigid (see Corollary 1 in \cite{FK} and its proof). Thus we get the following {\it local} result, in which the notation $|\cdot|_\delta$ stands for the sup norm inside the analyticity band of width $\delta$. 

\begin{cor} \label{local} For any $\a \in \R-\Q$ and any $\delta >0$, there exists $\eps(\delta, \a) >0$ such that if $f$ is a real analytic pseudo-rotation  with rotation number $\rho(f)=\a$ and if 
$$ |f-R_\a|_\delta  \leq \eps(\delta, \a)$$ 
then $f$ is $C^\infty$-rigid. 
\end{cor}

\begin{rema} Observe that $\eps(\delta,\a)$ essentially depends on $\delta$ and $ \sum \ln(q_{n+1})/q_n$ where $q_n=q_n(\a)$ and that $\eps(\delta,\infty)=\infty$ by Theorem \ref{Brjuno}.
\end{rema}
\medskip

We conclude this introduction with a series of comments and questions.

\subsection{Some questions around the rigidity of pseudo-rotations.}  

The following question was raised by Bramham in \cite{bramham}. 
\begin{question} Is every $C^k$ pseudo-rotation $f$ $C^0$-rigid? The question can be asked for any $k \geq 1$, $k=\infty$ or $k = \o$. 
\end{question}

In the case $k=\o$ or  $\rho(f)$ Diophantine and $k=\infty$, the latter question becomes an intermediate question relative to the Birkhoff-Herman problem on the conjugability of $f$ to the rigid disc rotation of angle $\rho(f)$ \cite{herman-ICM}.

As discussed earlier a better {\it a priori} control on the growth of $\|Df^{m}\|$ for a pseudo-rotation is sufficient to deduce rigidity for larger classes of rotation numbers. In the case of a circle diffeomorphism $f$  a gap in the growth is known to hold between exponential growth in the case $f$ has a hyperbolic periodic point or a growth bounded by $O(m^2)$ if not \cite{polterovich}. Does a similar dichotomy hold for area preserving disc diffeomorphisms?

If for example a polynomial bound holds on the growth of $\|Df^{m}\|$ for a smooth pseudo-rotation, then $C^\infty$-rigidity would follow for any Liouville rotation number by the same proof as that of Theorem \ref{Brjuno}.

\begin{question} \label{q.liouville} Is there any polynomial bound on the growth of the derivatives of a pseudo-rotation?  Is every $C^\infty$ pseudo-rotation with Liouville rotation number $C^0$ (or even $C^\infty$) rigid?
\end{question}

 With Herman's smooth version of R\"ussmann's Theorem \ref{russmann} for Diophantine rotation numbers (see \cite{FK}), a positive answer to the second part of Question \ref{q.liouville} would imply that smooth pseudo-rotations, and therefore  area preserving smooth diffeomorphisms of the disc with zero topological entropy  are never topologically mixing.

The local result of Corollary \ref{local} that holds for {\it every} analytic  pseudo-rotation raises the following natural question on its {\it semi-global} potential validity. 

\begin{question} Does the local result of Corollary \ref{local} hold for an $\eps$ that is independent on the rotation number of the pseudo-rotation $f$?
\end{question}

\medskip

\noindent {\bf Non-rigid $C^0$ pseudo-rotations.} In the case of homeomorphisms, observe that Crovisier  constructed on the sphere ${\mathbb S}^2$, $C^0$ pseudo-rotations of arbitrary rotation number  that have positive topological entropy \cite{crovisier}. It is simple to see that positive topological     entropy precludes $C^0$-rigidity. 

 Here is a direct construction of non rigid $C^0$ pseudo-rotations. Let $f_t$, $t\in [0,1]$ be a continuous family of circle homeomorphisms such that $f_0=R_\a$, $\rho(f_t)=\a$ for every $t \in [0,1]$, and $f_t$ is of class $C^2$ for every $t \in [0,1]-\{\frac{1}{2}\}$ while $f_{\frac{1}{2}}$ is a Denjoy counterexample. Hence $f_t$ is topologically conjugate to $R_\a$ for each $t \neq \frac 1 2$ by Denjoy Theorem while $f_{\frac 1 2}$ has a wandering interval. Then, we consider a homeomorphism of the disc $f$ that leaves each circle of radius $t \in [0,1]$ invariant and acts on this circle by the homeomorphism $f_t$. It is clearly not $C^0$-rigid since it is not rigid on the circle of radius $1/2$. On the other hand, $f$ leaves invariant the measure $\mu$ given by the integration of the invariant measures of $f_t$ on each circle of radius  $t$, and it is clear  that 
 $\mu(O)>0$ for any open set $O$, therefore $f$ can be conjugated by Oxtoby Ulam lemma to an area preserving homeomorphism  that will be a non $C^0$-rigid pseudo-rotation of angle $\a$. These examples fail however to be mixing. 

\begin{question} \label{C0mixing} Does there exist a mixing $C^0$ pseudo-rotation of the disc? 
\end{question}

\medskip

\noindent{\bf Pseudo-rotation of the annulus.}  Pseudo-rotations of the annulus $\AAA=\T \times [0,1]$ are area preserving maps of $\AAA$ that preserve the boundaries and have a unique  irrational rotation number (in projection on the first variable). 
\begin{question} Are pseudo-rotations of the annulus rigid?
\end{question}

Theorems \ref{Brjuno} and \ref{superliouville} hold on the Annulus, with the same proof as for the Disc. But even in the real analytic category, we do not have a proof of Theorem \ref{zeroentropy}. Indeed,  
the proof of the KAM Theorem \ref{russmann} can be adapted to the annulus only if the map is supposed to be analytically conjugated to $R_\a$ on the boundary. But to guarantee that this always holds one has to ask that $\a$ satisfies the so called condition $\cH$ \cite{yoccoz.analytic}. Hence, if $\a$ does not satisfy $\cH$ but satisfies the Brjuno condition (see \cite[Example 2.13]{yoccoz.analytic}) 
 then neither Theorem \ref{main-rigid} nor Theorem \ref{russmann} can be used, and we do not know if absence of mixing always holds for real analytic pseudo-rotation of the annulus with rotation number $\a$.

\medskip 

\noindent{\bf  The non-conservative case.} Similar results as in Theorems \ref{Brjuno} and \ref{superliouville}  hold if instead of Lebesgue measure we suppose that $f$ preserves any measure that assigns a lower bounded mass as a function of $r$ to discs of radius $r$. Thus a natural question is the following. 
\begin{question} \label{q8} Is it true that a 
diffeomorphism $f$ of the disc with zero topological entropy that has a unique rotation number is  never topologically mixing?  
\end{question}

Since Franks and Handel classification result holds for diffeomorphisms that preserve a measure that is strictly positive on every open set, a positive answer to Question \ref{q8} 
 would imply a positive answer to the first part of the following question.

\begin{question} Is it true that a 
diffeomorphism of the disc with zero topological entropy that preserves a measure that is strictly positive on every open set is  never topologically mixing?    Is it true that a 
diffeomorphism of the disc with zero topological entropy is  never topologically mixing?   
\end{question}

\medskip 

\section{Rotation number, Recurrence, and Rigidity} \label{sec.rigidity}

\subsection{Rotation number and recurrence}
We give here the main ingredient of our approach to rigidity of pseudo-rotations. It is a simple observation that combines Franks' free disc lemma with Kac's lemma to give a lower bound on the measure of 
free topological discs of a pseudo-rotation with a small rotation number.
\begin{lemma} \label{main-rigid} Let $f$ be a continuous pseudo-rotation of $\DD$. If $\rho(f)=\eps+\ZZ$, then any topological disc $D$  such that $\l(D) > \eps$ satisfies $f(D) \cap D \neq  \emptyset$. 
\end{lemma}

\begin{proof}
Let $D$ be a disc such that $f(D) \cap D = \emptyset$ and $\tilde D$ a fixed connected component of the lift of $D$ to
$\widetilde A=\R\times[0,1]$, the universal covering space od $\DD^2\setminus\{0\}$. Since $f $ does not have periodic points on $\DD^2-\{0\}$, neither does the lift $\tilde{f}$ to $\tilde{\AAA}$, whose (real) rotation number $\rho(\widetilde f)$ is $\eps$. In particular $\tilde{f}$ is fixed point free. Note that $\tilde{f}(\tilde D) \cap \tilde D = \emptyset$. Hence, by Franks' Theorem on positively and negatively  returning discs \cite[Theorem 2.1]{franks} we have that either $\tilde{f}^n(\tilde D) \cap (\tilde D+l) =\emptyset$ for every $n \geq 1, l \geq 0$ or $\tilde{f}^n(\tilde D) \cap (\tilde D+l) =\emptyset$ for every $n \geq 1, l \leq 0$. For definiteness we assume the second condition holds. By Poincar\'e recurrence, $\lambda$-a.e. $x \in D$ returns infinitely many times to $D$ under iteration by $f$, call $n_D(x)$ the first strictly positive return time of such a point and $f_D$ the first return map $f_D(x)=f^{n_D(x)}(x)$. Write $l_D(x)$ for the integer such that $\tilde f^{n_D(x)}(\tilde x)\in \tilde D+l_D(x)$, where $\tilde x$ is the lift of $x$ belonging to $\tilde D$. By our assumption we know that $l_D(x)\geq 1$. Next, since $f$ is a pseudo-rotation and $\rho(\tilde f)=\eps$, we necessarily have {\it for every} $x$
 \begin{equation} \lim_{N \to \infty}  \frac{l_D(x)+l_D(f_D(x))+\ldots+l_D(f_D^{N-1}(x))}{n_D(x)+n_D(f_D(x))+\ldots+n_D(f_D^{N-1}(x))} = \eps, \label{eq.ret}\end{equation}
  which implies
\begin{equation} \liminf_{N \to \infty}  \frac{n_D(x)+n_D(f_D(x))+\ldots+n_D(f_D^{N-1}(x))}{N} \geq \frac 1\eps .\label{eq.ret2}\end{equation}
 Now, Kac's Lemma asserts that 
\begin{equation} \int_D n_D(x) d\lambda =\lambda\left(\bigcup_{n\geq 0} f^n(D)\right)\leq1. \label{eq.kac} \end{equation}
Finally, by invariance of the restriction of $\lambda$ to $D$ by $f_D$ and Fatou's Lemma we get from \eqref{eq.ret2} and \eqref{eq.kac} 
\begin{multline*} 1 \geq  \int_D \liminf_{N \to \infty}  \frac{n_D(x)+n_D(f_D(x))+\ldots+n_D(f_D^{N-1}(x))}{N} d \lambda \\ \geq \int_D \frac 1\eps d \lambda = \frac{\lambda(D)}{\eps}. \end{multline*}
 \end{proof}

An immediate consequence of Lemma \ref{main-rigid} is the following estimate on the maximal displacement of a pseudo-rotation with respect to its rotation number. 

\begin{cor} \label{cor.displacement} Let $f$ be a continuous pseudo-rotation of $\DD$ with $\rho(f)=\eps+\ZZ$. Then 
\begin{equation*} \|f-\Id\|_0  \leq \eps^{\frac{1}{2}} +  \max_{x \in \DD} {\rm diam} \left(f(B(x,\eps^{\frac{1}{2}})) \right) \end{equation*}
\end{cor}

\subsection{Rigidity of super-Liouville  pseudo-rotations}

\begin{proof}[Proof of Theorem \ref{superliouville}]

Just observe that if $f$ is H\"older with exponent $a$ then there exists $C$ such that for any $m$  
 \begin{equation} \label{eqholder} |f^{m}(x)-f^{m}(y)|\leq C^{m} |x-y|^{a^{m}} \end{equation}
and apply Corollary  \ref{cor.displacement} to the pseudo-rotation $f^{q_{n_j}}$ that satisfies      $\rho(f^{q_{n_j}})= \|q_{n_j} \a \| \leq \frac{1}{q_{{n_j}+1}}$. Hence 
$$\|f^{q_{n_j}}-\Id\|_0 \leq  \frac{1}{\sqrt{q_{{n_j}+1}}} + C^{q_{n_j}}  \left(\frac{2}{\sqrt{q_{{n_j}+1}}}\right)^{a^{q_{n_j}}}$$ 
and rigidity follows from the arithmetic condition   $$\lim {q_{n_j}}^{-1}  a^{q_{n_j}} \ln q_{n_j+1} = +\infty. $$
  \end{proof}

\subsection{Rigidity of non-Brjuno type pseudo-rotations}

\begin{proof}[Proof of Theorem \ref{Brjuno}]

We will need the following result on the growth of the denominators of a non Brjuno type number.

\begin{lemma} \label{lemma.nonb}�Suppose $\a \in \R-\Q$ is not of Brjuno type. 
 For any $H>1$, there exists  a subsequence $q_{n_j}$ of the sequence $q_n(\a)$ such that $q_{n_{j+1}} \geq H^{q_{n_j}}$ and  there exists an infinite set $\cJ$ such that 
\begin{equation} \label{nonB}  \forall j \in \cJ, \quad  \| q_{n_j} \a \| < e^{-\frac{q_{n_j}}{j^2}} \end{equation}
\end{lemma}

\begin{proof} Let $q_{m_j}$ be such that $q_{m_{j+1}-1} \leq H^{q_{m_j}}$ while $q_{m_{j+1}} \geq H^{q_{m_j}}$. The non Brjuno condition implies that for an infinite set $\cI \in \N$ we have that 
\begin{equation} \label{nonB.finite} \sum_{m=m_j}^{m_{j+1}-1} \frac{\ln q_{m+1}}{q_m}  \geq \frac{1}{j^{\frac{3}{2}}}.\end{equation}
We claim that for $j \in \cI$, we have that there exists $l_j \in [m_j, m_{j+1}-1]$ such that       $q_{l_j+1} \geq  e^{\frac{q_{l_j}}{j^{1.6}}}$. Before we prove this claim, let us observe how it leads to the conclusions of the lemma. Since  $q_{m_{j+1}} \geq H^{q_{m_j}}$ for any $j \in \N$, we have that $q_{l_{j+2}} \geq H^{q_{l_j}}$ for any $j \in \N$. Assume now that $\cI$ contains infinitely many even integers (the other case with infinitely many odd integers being similar), and let $n_j:=l_{2i}$. The infinite set $\cJ$ of integers $j \geq 100$ such that $2j \in \cI$ satisfies the conditions of the lemma. Indeed, $q_{n_{j+1}}=q_{l_{2(j+1)}} \geq H^{q_{l_{2j}}}=H^{q_{n_j}}$, and for $j \in \cJ$, $  \| q_{n_j} \a \| \leq 1/q_{n_j+1} \leq      e^{-\frac{q_{n_j}}{(2j)^{1.6}}} \leq  e^{-\frac{q_{n_j}}{j^2}} $ since $j \geq 100$.

To prove the claim, observe first that since  $q_{m_{j+1}-1} \leq H^{q_{m_j}}$, then we have that there exists at most two indices $l,l' \in [m_j,m_{j+1}-2]$ such that $\frac{\ln q_{h+1}}{q_h} \geq \frac{1}{\sqrt{q_h}}$, for $h=l,l'$. For $j$ sufficiently large, since $ \sum_{m=m_j}^{m_{j+1}-2} \frac{1}{\sqrt{q_m}} \ll \frac{1}{j^2}$ (because  $q_{m+2}  \geq 2 q_m$ for any $m$ and $m_j \geq j$) then at least one of $h=l$ or $h=l'$ or $h=m_{j+1}-1$ must satisfy $\frac{\ln q_{h+1}}{q_h} \geq  \frac{1}{j^{1.6}}$ as claimed. \end{proof} 


Back to the proof of Theorem \ref{Brjuno}, fix $\a \in \R-\Q$ that is not of Brjuno type. Let $q_{n_j}$ be as in Lemma \ref{lemma.nonb}   where $H$ is given by Theorem \ref{1.corollary 1.1} depending on the $C^2$ norm of $f$.

Since $f$ 
has no hyperbolic periodic point, Theorem \ref{1.corollary 1.1} implies that 
\begin{equation} \label{katokC1}  \| D f^{q_{n_j}} \| < e^{\theta^j {q_{n_j}}} \leq e^{ \frac{q_{n_j}}{j^3}} \end{equation}
if $j$ is sufficiently large. 

But $\rho(f^{q_{n_j}})= \|q_{n_j} \a\|$, hence by Corollary \ref{cor.displacement}, \eqref{nonB} and \eqref{katokC1} we get  for $j \in \cJ$ sufficiently large that
$$\| f^{q_{n_j}} -\Id\|_0 \leq e^{-\frac{q_{n_j}}{3 j^2}}$$

Since \eqref{katokC1} also implies that  
\begin{equation} \label{katokCk} \| D^k f^{q_{n_j}} \| < e^{C_{k} \frac{q_{n_j}}{j^3}} \end{equation}
we get by Hadamard convexity norm estimates 
$$\| f^{q_{n_j}} -\Id\|_{k-1} = o(e^{-\frac{q_{n_j}}{j^3}}).$$ \end{proof} 


\section{A growth gap for area preserving surface diffeomorphisms}\label{katok}

\subsection{A finite information closing lemma and the gap in the derivatives growth}

Let $S$ be a compact smooth surface with a Riemannian metric.
Denote ${\rm Diff}^{r}_{\text{vol}}(S)$ the group of $C^r$ diffeomorphisms which preserve the volume form induced by the Riemannian metric.

Katok's closing lemma asserts that if a $C^{1+\epsilon}$ diffeomorphism $f$ of a compact manifold has a hyperbolic invariant  probability measure (with non-zero Lyapunov exponents) then $f$ has (many) hyperbolic periodic points. As a consequence, we get that if $f \in {\rm Diff}^{r}_{\text{vol}}(S), r>1$, and $f$ has no hyperbolic periodic points then  the sequence of sup norms of the differnetials $\|Df^n\|$ grows sub-exponentially.  Indeed, if the contrary holds, one can easily construct a hyperbolic invariant probability measure. Hence, if one considers a compact set $K$ inside the set of ${\rm Diff}^{r}_{\text{vol}}(S)$ for say $r=2$, then there must be an {\it a priori} gap between exponential and sub-exponential growth of $\|Df^n\|$. Our goal here is to give an explicit estimate on this gap. 

We first give a finite information version of Katok's closing lemma for a map $g\in {\rm Diff}^{2}_{\text{vol}}(S) $ that only requires a growth of $\norm{Dg^q}$ comparable to $\norm{Dg}^{\theta q}$ where  $ \theta$ is close to $1$ and  $q$ is sufficiently large compared to powers of the $C^2$ norm of $g$. We recall that the $C^r$ norm of a diffeomorphism of a smooth compact manifold $M$  can be defined by taking the supremum over all (finitely many) charts of the $C^r$ norms of the corresponding maps from $\R^d$ to $\R^d$ where $d$ is the dimension of $M$.

\begin{theo} \label{main prop}
There exist constants $A_0>0, \theta_0\in (0,1), H_0>0$ such that for all $(g, q, A, D, \theta)$ with $g\in {\rm Diff}^{2}_{\text{vol}}(S) , q\in \mathbb{N}, A\geq A_0, \theta\in[\theta_0,1), D\geq A, q\geq D^{H_0}$, if
\begin{eqnarray}
\norm{Dg} &\leq& A \label{p1}\\
\norm{D^2g}&\leq& D\label{p2}\\ 
\norm{Dg^q}&\geq& A^{\theta q} \label{p3}
\end{eqnarray}
Then $g$ has a hyperbolic periodic point.
\end{theo}

Let us first see how the gap on the growth of $\|Df^n\|$ announced in Theorem \ref{1.corollary 1.1} follows from Theorem \ref{main prop}.

\begin{proof}[Proof of Theorem \ref{1.corollary 1.1}]
Consider $K'\subset K$ defined as 
$$K'=\{f\in K: f \text{ has no hyperbolic periodic point } \}$$
 By stability of hyperbolic periodic points, ${K'}^C$ is open, so $K'$ is compact. We claim that
\begin{equation} \label{claim1} \lim_{n\to \infty}\sup_{x\in S, f\in K'}\frac{1}{n}\log\norm{Df^n(x)}=0.\end{equation} 

\begin{proof}[Proof of \eqref{claim1}]
Let $\alpha_n:= \sup_{x\in S, f\in K'} \log\norm{Df^n(x)}$,  we have for all $m,n\in\mathbb{N}$,$$\alpha_n+\alpha_m\geq \alpha_{m+n}, \alpha_n\geq 0.$$
Let $A_0, \theta_0, H_0$ be given in Theorem \ref{main prop}. From the sub-additive lemma, we have $$c:=\lim_{n\to\infty}\frac{\alpha_n}{n}=\inf\frac{\alpha_n}{n}\geq 0.$$ We have to show that $c=0$. Suppose the contrary is true, then there exists an $n\in \mathbb{N}$ large enough such that $\alpha_n\leq cn\theta_0^{-1}, e^{cn\theta_0^{-1}}\geq A_0$. 

From the compactness of $K$, we know that there exists a constant $C\geq 1$ such that for all $n\in \mathbb{N}, f\in K$,$\norm{D^2(f^n)}\leq e^{Cn}$. Let $$A:=e^{cn\theta_0^{-1}}, D:=\max(e^{Cn}, A),\theta:=\theta_0, q:= D^{H_0}.$$ Since $\frac{\alpha_{qn}}{qn}\geq c$, we can find an $f_0\in K'$ such that $\norm{Df_0^{qn}}\geq e^{cqn}$, let $g:=f_0^n$. Then such $(g,q, A,D,\theta)$ satisfies the conditions of Theorem \ref{main prop}, therefore by Theorem \ref{main prop} we have that $f_0$ has a hyperbolic periodic point, which contradicts the definition of $K'$. Hence \eqref{claim1} is proved. \end{proof} 

As a consequence of  \eqref{claim1}, there exists $H'$ such that for all $H > H'$
\begin{eqnarray}
\sup_{x\in S, f \in K'}\frac{1}{H}\log\norm{Df^{H}(x)}\leq 1. \label{1.5}
\end{eqnarray}
On the other hand, for $A_0, \theta_0$ as in Theorem \ref{main prop}, there exists a positive integer $H''$ sufficiently large such that for any sequence $\{q_n, n\geq 0\}$ satisfying $q_0\geq H'', q_n\geq {H''}^{q_{n-1}}$, for each $n>0$, we have 
\begin{eqnarray}
e^{\theta_0^{n-1}q_{n-1}}>A_0 \label{H''}.
\end{eqnarray}
Indeed  for large $H''$, $q_n$ increases much faster than $\theta_0^{-n}$.

Choose $H\geq \max(H', H'', e^{2CH_0}), \theta\geq\theta_0$, where $\theta_0, H_0$ are given in Theorem \ref{main prop}. We claim that such $(H,\theta)$ satisfies our conditions.

If there exist $f\in K, n\geq 0$ satisfying \eqref{cr1.1}, where $q_0\geq H, q_n\geq H^{q_{n-1}}$ as in Theorem \ref{1.corollary 1.1}. We take $n$  to be the smallest integer satisfying \eqref{cr1.1}.

If $n=0$, there thus exists an $x\in S$ such that $\frac{1}{q_0}\log\norm{Df^{q_0}(x)}>1$.  Since $q_0\geq H$,  by \eqref{1.5} we know $f\notin K'$. Therefore $f$ has a hyperbolic periodic point.

If $n>0$, then let $g:= f^{q_{n-1}}, A:=e^{\theta^{n-1}q_{n-1}}, q:=\frac{q_n}{q_{n-1}}, D:=e^{Cq_{n-1}}$. By our assumptions on $H$, $\theta$, $H''$ and \eqref{H''}, we have $A>A_0$. It is clear that $(g, A, \theta, q, D)$ satisfy the conditions in Theorem \ref{main prop}.

In fact, since $n$ is the smallest integer satisfying \eqref{cr1.1}, we get \eqref{p1} and \eqref{p3} in Theorem \ref{main prop}.  We have (2) since $$\norm{D^2(g)}\leq e^{Cq_{n-1}}=D, \log q \geq q_{n-1}\log H-\log q_{n-1}$$
Finally,  $\log H\geq 2CH_0$, so we have $\log q \geq 2CH_0 q_{n-1}-\log q_{n-1}\geq CH_0q_{n-1}$, therefore $q\geq D^{H_0}$. 

From Theorem \ref{main prop}, we know that $g$ has a hyperbolic periodic point, then $f$ also has a hyperbolic periodic point.
\end{proof}

\noindent {\bf Strategy of the proof of Theorem \ref{main prop}.}
Let us briefly recall how Katok derives his closing lemma from Pesin theory. For almost every point of the hyperbolic measure,  \textup{Lyapunov neighborhoods} are constructed with stable and unstable manifolds attached to the point, which have definite size under the \textup{Lyapunov metric}. Moreover, the transformation between two neighborhoods looks like a uniform hyperbolic map under these new metrics. Even though the Lyapunov neighborhoods of some points are very small, Pesin theory provides sets of positive measure, the so called Pesin basic sets, consisting of points having Lyapunov neighborhoods (and thus stable and unstable manifolds) of uniformly lower bounded sizes and having as well a lower bounded angle between the stable and unstable directions. Then, the Poincar\'e recurrence theorem is used to obtain an integer $L$  and a  point  $x$  from a given Pesin basic set that comes back under $L$ iteration inside the same basic set close to itself. This then guarantees the existence of a hyperbolic periodic point $z$ of period $L$ that shadows the piece of orbit of $x$ of length $L$, in the same way as Anosov closing lemma gives a hyperbolic periodic point that shadows a recurrent piece of orbit in uniform hyperbolic dynamics. The proof can for example be done through graph transforms involving the first return map to the Lyapunov neighborhood of $x$. Of course, once the good Pesin point $x$ is detected, the rest of the proof does not depend on informations on $f$ beyond the ones we have for the first $L$ iterates of a small  neighborhood of $x$.  
 The problem with this procedure when one is interested in finite information versions of the closing lemma, lies actually in the fact that the Lyapunov charts construction requires an infinite amount of information involving the full orbit of points.

{ In the spirit of the finite information versions of the closing lemma based on the concept of effective hyperbolicity at a point of  \cite{CP},  our proof of Theorem \ref{main prop} uses properties \eqref{p1}--\eqref{p3} to select a good point for which nice hyperbolic properties hold most of the time along a piece of its orbit until it comes back very close to itself, and then deduce from a hyperbolic-like  property of the return map  the existence of a hyperbolic periodic point next to the selected good point. 

More precisely, our proof of Theorem \ref{main prop} goes as follows :  In Section \ref{finding orbit}, we use \eqref{p3} to find "good points" $x\in S$ for which there exists  $L\in \{1,\dots q\}$ and $v_s, v_u\in T_xS$, such that  the forward orbit from $x$ to $g^L(x)$ contracts and expands consistently at rate almost $A$ the directions $v_s$ and $v_u$ respectively (this is possible because $\theta$ is close to $1$), and the backward orbit by $g^{-1}$ from $g^L(x)$ to $x$ expands and contracts consistently at rate almost $A$ in the directions $Dg^L(v_s)$ and $Dg^L(v_u)$ respectively.  To neutralize the nonlinearities coming from the second derivatives, we also 
ask the distances $d_{T_1S}(v_s,\frac{Dg^L(v_s)}{\norm{Dg^L(v_s)}})$ and $d_{T_1S}(v_u,\frac{Dg^L(v_s)}{\norm{Dg^L(v_s)}})$ to be small compared to a power of $D^{-1}$. Here $d_{T^1S}$ is a metric on the unit tangent bundle of $S$ defined by some embedding into a Euclidean space (in particular, the point $g^L(x)$ is very close to $x$). All the former is possible to achieve due to Pliss Lemma and the pigeonhole principle (the hypothesis $q \geq D^{H_0}$ is crucial here). From area conservation, we also get {\it for free}  that the angles $\angle(v_s, v_u)$ and $\angle(Dg^L(v_s), Dg^L(v_u))$ are not too small. We do not assume any {\it a priori }�control on the angles of the expanded and contracted direction within the $L$-orbit.

Next we study the dynamics along the $L$-orbit of the good point $x$ and show that it is possible to find a box $B$ (think of a square) in the neighborhood of $x$ that contains a vertical strip that is mapped by the return map $g^L$ into a horizontal strip with in addition a strict cone contraction condition for $g^L$ and $g^{-L}$ inside these strips. This evidently implies the existence of a hyperbolic fixed point for $g^L$.  }

\subsection{Finding a suitable finite orbit}\label{finding orbit}



We first precise the notion of good points mentioned in the introduction.

\begin{definition} \label{2.def of stable unstable direction}
For all triple $(x, v_s, v_u)$ such that $x\in S, v_s, v_u\in T_xS$, for all $i\in \mathbb{Z}$ we denote
$$v^s_i:=\frac{Dg^{i}(v_s)}{\norm{Dg^{i}(v_s)}}, v^u_i:=\frac{Dg^{i}(v_u)}{\norm{Dg^{i}(v_u)}}, $$ $$\lambda^s_{i}:=\log\frac{\norm{Dg(v^s_i)}}{\norm{v^s_i}}, \lambda^u_{i} :=\log\frac{\norm{Dg(v^u_i)}}{\norm{v^u_i}},$$ $$\overline{\lambda}^e_i:=\min\{\lambda^u_i,-\lambda^s_i\}$$

\end{definition}

\begin{definition}(Forward and backward $(L,a)-$good triple) \label{forward backward good triple}

For all $L\in \mathbb{N}, a>0$, we say a triple $(x, v_s, v_u)$ is forward $(L,a)-$good if the exponents $\lambda_i^{s,u}, \overline{\lambda}_i^{e}$ associated to $(x, v_s, v_u)$ satisfy the following inequalities:
\begin{eqnarray}
|\lambda_j^{s,u}|&\leq& a, \forall 0 \leq j \leq L-1 \\ \label{5y}
\frac{1}{k}\sum_{j=0}^{k-1}\overline{\lambda}_{j}^e&>&(1-\frac{1}{1000})a, \forall 1\leq k\leq L.\label{bar frwd gd}
\end{eqnarray}

We say a triple $(x, v_s, v_u)$ is backward $(L,a)-$good if the exponents $\lambda_i^{s,u}, \overline{\lambda}_i^{e}$ associated to $(x, v_s, v_u)$ satisfy the following inequalities:
\begin{eqnarray}
|\lambda_{j}^{s,u}|& \leq& a, \forall -L\leq j \leq -1\\
\frac{1}{k}\sum_{j=-k}^{-1}\overline{\lambda}_{j}^e&>&(1-\frac{1}{1000})a, \forall 1\leq k\leq L.\label{bar bckwd gd}
\end{eqnarray}
\end{definition}

For later use, we denote 
\begin{eqnarray} \label{label lambda e}
\lambda^e_i:=\min\{ \lambda^u_i, \lambda^u_i-\lambda^s_i, -2\lambda^s_i \}
\end{eqnarray}
We have the following inequalities for $\lambda^{e,s}_i$ associated to forward and backward $(L,a)-$good triples.
\begin{lemma}\label{l2.1}For a forward $(L,a)-$good triple $(x, v_s, v_u)$, the exponents $\lambda_i^{e, s}$ associated to $(x, v_s, v_u)$ satisfy the following inequalities:
\begin{eqnarray}
\frac{1}{k}\sum_{j=0}^{k-1}\lambda_{j}^s&<&-\frac{1}{2} a, \forall 1\leq k\leq L,\label{5x}\\ 
\frac{1}{k}\sum_{j=0}^{k-1}\lambda_j^e&>&(1-\frac{1}{100})a, \forall 1\leq k\leq L,\label{tmp 1}\\
\frac{1}{k}\sum_{j=0}^{k-1}\min(3\lambda_{j}^e,0)&>&-\frac{1}{10}a, \forall 1\leq k\leq L.\label{7x}
\end{eqnarray}
For a backward $(L,a)-$good triple $(x, v_s, v_u)$, the exponents $\lambda_i^{s,e}$ associated to $(x, v_s, v_u)$ satisfy the following inequalities:
\begin{eqnarray}
\frac{1}{k}\sum_{j=-k}^{-1}\lambda_{j}^s&<&-\frac{1}{2} a, \forall 1\leq k\leq L, \label{backward s-good}\\
\frac{1}{k}\sum_{j=-k}^{-1}\lambda_j^e&>&(1-\frac{1}{100})a, \forall 1\leq k\leq L.\label{tmp -1}
\end{eqnarray}

\end{lemma}

\begin{proof}
Notice that $\lambda^s_i \leq-\overline{\lambda}^e_i$, the inequalities \eqref{5x}, \eqref{backward s-good} are immediate consequences of \eqref{bar frwd gd} and \eqref{bar bckwd gd}.

Consider $\lambda^e_i-\overline{\lambda}^e_i$, it is nonnegative when $\lambda_i^s\leq 0$, otherwise it larger or equal than $-3a$. But by the definition of forward $(L,a)-$good point, for all $1\leq k \leq L$, there are at most $\frac{1}{1000}k$ many $j$ such that $0 \leq j\leq k-1$ and $\overline{\lambda}^e_j<0$. Since $\lambda^s_i \leq -\overline{\lambda}^e_i$ for all $0 \leq i\leq L-1$, we have there are most $\frac{1}{1000}k$ many $j$ such that $0 \leq j \leq k-1$ and $\lambda_j^s>0$. Then we have for all $1\leq k \leq L$, $$\frac{1}{k}\sum_{j=0}^{k-1}\lambda_j^e-\overline{\lambda}_j^e\geq \frac{1}{1000}\cdot (-3a) > -\frac{1}{200}a$$
By \eqref{bar frwd gd}, we get \eqref{tmp 1}. The proof of \eqref{tmp -1} is similar.

Now we prove \eqref{7x}. Since $\lambda_j^e\leq\lambda_j^u\leq a$ and we have already proved \eqref{tmp 1}, there are at most $\frac{1}{100}k$ many $j$ such that $\lambda_j^e \leq 0$, for which we have $\lambda^e_j \geq -2a$. Then we get $$\frac{1}{k}\sum_{j=0}^{k-1}\min(3\lambda_{j}^e,0)\geq\frac{1}{100}(-6a)>-\frac{1}{10}a, \forall 1\leq k\leq L.$$
This proves \eqref{7x} and thus completes the proof.
\end{proof}

\begin{rema}
Without loss of generality, we fix an isometric embedding $\Psi: S \to \mathbb{R}^{W}$. This allows us to define the angle between two arbitrary tangent vectors by translating them back to $0$. As a result, we can define the distance between any two unit tangent vectors of the surface $v_1\in T^1_{x_1}S, v_2\in T^1_{x_2}S$ as the following: 
\begin{eqnarray*}
d_{T^1S}(v_1,v_2)=\angle(v_1, v_2)+d(x_1,x_2)
\end{eqnarray*}
In particular, for $v_1\in T^1_{x_1}S, v_2\in T^1_{x_2}S, d_{T^1S}(v_1,v_2)\geq d(x_1,x_2)$.
\end{rema}

\begin{definition}\label{(q,a)-good point}((q,a)-good point)

For all $q\in \mathbb{N}, a>0$, we say a point $x\in S$ is $(q,a)-$good if there exist $v_s, v_u\in T_xS, L\in \mathbb{N}$ such that the following is true. Let $v^s_i, v^u_i, i \in \Z$ be defined in Definition \ref{2.def of stable unstable direction} associated to the triple $(x, v_s, v_u)$. Then\\
\begin{itemize} 
\item[(1)] The triple $(x, v_s, v_u)$ is forward $(L,a)-$good.\\
\item[(2)] The triple $(g^L(x), v^s_L, v^u_L))$ is backward $(L,a)-$good.\\
\item[(3)]  It holds $$\log|\cot\angle(v_s, v_u)|\leq 3a$$ $$\log|\cot\angle(v^s_L, v^u_L)|\leq 3a$$\\
\item[(4)] It holds  $$d_{T^1S}(v_s, v^s_L)<q^{-\frac{1}{100}}$$ $$d_{T^1S}(v_u, v^u_L)<q^{-\frac{1}{100}}$$
\end{itemize} 
\end{definition}
 
The goal of this section is to prove the following proposition which shows the existence of a $(q,a)-$good point under certain conditions. We will then show in the Section \ref{construct graph transf} that the existence of a $(q,a)-$good point for sufficiently large $q$ ( the largeness only depends on $a$ and $S$ ) implies the existence of a hyperbolic periodic point.
\begin{prop} \label{main prop in sec 2}
There exists $\theta_0\in(0,1)$,$A_0>0,H_0 >0$, such that the following holds. If $(g, A, \theta, q)$ satisfies $ A>A_0, \theta\in [\theta_0,1)$, $q \geq A^{H_0}$  and conditions \eqref{p1},\eqref{p3} in Theorem \ref{main prop}, then there exists a $(q,a)-$good point $x\in S$, where $a=\log A$.
\end{prop}

\begin{proof}
We will use the following lemma several times:
\begin{lemma}(Pliss  lemma \cite{pliss})
Given a sequence of $n$ real numbers $a_1,...,a_n$. Assume $a_i \leq l$ for all $1\leq i \leq n$, and $\sum_{i=1}^{n}a_i > nl'$, $l'<l$.Then for any $l''<l'$, there exist at least $\frac{l'-l''}{l-l''}n$ many $i$'s such that $\frac{1}{k}\sum_{j=i}^{i+k-1} a_j > l''$ for all $k$ satisfying $i+k-1\leq n$.
\end{lemma}

Since $(g, A, \theta, q)$ satisfies \eqref{p1} and \eqref{p3} in Theorem \ref{main prop}, we know there exists $x \in S$ such that for  $a=\log A$
\begin{eqnarray}\label{10x}
\frac{1}{q}\log\norm{Dg^q(x)} \geq \theta a
\end{eqnarray}
We claim that if $A_0, H_0$ are large enough, $\theta_0$ is sufficiently close to $1$, then there exists a $(q,a)-$good point in the orbit of any $x$ satifying \eqref{10x}.

Let $v_0^s$ be a unit vector in the most contracting direction of $Dg^q(x)$ in $T_xS$, and let $v_0^u = {v_0^s}^{\perp}$  denote a unit vector orthogonal to $v_0^s$.

From now on to the end of this subsection, we consider $v_i^s, v_i^u, \lambda_i^{u,s}, \overline{\lambda}_i^{e}$ as in Definition \ref{2.def of stable unstable direction} for the triple $(x, v_0^s, v_0^u)$, where $x$ satifies \eqref{10x}. Then we have the following estimates for $\lambda_i^{u, s},  \overline{\lambda}_i^{e}$ which will be used later. 

\begin{lemma} \label{l2.2}
For any $0 < \eta < 1$, there exists $\theta_0\in (0,1)$ depending only on $\eta$ such that the following is true. 
Suppose $g\in  {\rm Diff}^1_{\text{vol}}(S)$ satisfies  $\norm{Dg}\leq A$. If $\theta\in[\theta_0, 1)$ and $(q,a,\theta, x)$ satisfies \eqref{10x}, where $a=\log A$, then we have for all $0 \leq i \leq q-1$, $$|\lambda^s_i|\leq a, |\lambda^u_i |\leq a$$ Moreover we have the following inequality:
$$\frac{1}{q}\sum_{i=0}^{q-1}\overline{\lambda}_i^e\geq (1-\eta)\theta a$$
\end{lemma}

\begin{proof}

Since $\norm{Dg} \leq A$,  we have $|\lambda_i^u|,|\lambda_i^s|\leq a$.

Since $v^s_0$ is the most contracting vector for $Dg^q(x)$, by \eqref{10x}, we have
\begin{eqnarray*}
\frac{1}{q}\sum_{i=0}^{q-1}\lambda_i^s &=& -\frac{1}{q}\log||Dg^q(x)|| \leq -\theta a\\
\frac{1}{q}\sum_{i=0}^{q-1}\lambda_i^u &=& -\frac{1}{q}\sum_{i=0}^{q-1}\lambda_i^s \geq \theta a
\end{eqnarray*}
The second line follows from the fact that $g$ preserves area.

As a result, we know that there are at most $\frac{4(1-\theta)}{\eta}q$ many $i$ such that $\lambda_i^u<(1-\frac{\eta}{2}) a$ or $\lambda_i^s>-(1-\frac{\eta}{2})a$, for which $i$ we have $\overline{\lambda}^e_i \geq -a$. For the rest of $i$, $\overline{\lambda}_i^e \geq (1-\frac{\eta}{2}) a$. Hence 
$$\frac{1}{q}\sum_{i=0}^{q-1}\overline{\lambda}_i^e\geq (1-\frac{\eta}{2})a-\frac{4(1-\theta)}{\eta}\times2a\geq a(1-\frac{\eta}{2} - \frac{8(1-\theta)}{\eta}). $$
If $\theta_0$ is sufficiently close to $1$ depending only on $\eta$, we get that  $$\frac{1}{q}\sum_{i=0}^{q-1}\overline{\lambda}_i^e\geq (1-\eta)\theta a$$\end{proof}

From now on to the end of the proof of Proposition \ref{main prop in sec 2}, we will also use the following definition.
\begin{definition} \label{2.temp def of good point}
The point $g^n(x)$ is called \textit{good} in the orbit of $x$ if $n \in [1,q-1]$ satisfies the following conditions:\\
\begin{eqnarray}
\frac{1}{k}\sum_{j=n}^{n+k-1}\overline{\lambda}_{j}^e&>&(1-\frac{1}{1000})a, \forall 1\leq k\leq q-n,\label{fwd good 2}\\
\frac{1}{k}\sum_{j=n-k}^{n-1}\overline{\lambda}_{j}^e&>&(1-\frac{1}{1000})a, \forall 1\leq k\leq n.\label{bwd good}
\end{eqnarray}

\end{definition}

We will use the following lemma to estimate the number of the good points:
\begin{lemma}\label{half pts r good}

There exists $\theta_0\in(0,1)$ such that if $\theta\in[\theta_0,1)$, $g\in {\rm Diff}^1_{\text{vol}}(S)$ satisfies $\norm{Dg}\leq A$, and $(q,a,\theta, x)$ satisfies \eqref{10x}, where $a=\log A$, then there are more than $\frac{q}{2}$ points in $\{g^k(x), 0\leq k\leq q-1\}$  that are good in the orbit of $x$.
\end{lemma}

\begin{proof}
By Lemma \ref{l2.2}, for any $\eta > 0$ for  we have $$\frac{1}{q}\sum_{i=0}^{q-1}\overline{\lambda}_i^e\geq (1-\eta)\theta a, \overline{\lambda}_i^e\leq a, \forall 0\leq i \leq L-1$$if we let $\theta$ to be close to 1. Using Pliss lemma again, we know that there are at least $\frac{(1-\eta)\theta-\frac{999}{1000}}{1-\frac{999}{1000}}q$ points $g^n(x)$ in $\{g^k(x), 0\leq k\leq q-1\}$ such that, $$\frac{1}{k}\sum_{j=n}^{n+k-1}\overline{\lambda}_{j}^e>(1-\frac{1}{1000})a, \forall 1\leq k\leq q-n.$$
Similarly, we know that there are at least $\frac{(1-\eta)\theta-\frac{999}{1000}}{1-\frac{999}{1000}}q$ points $g^n(x)$ in $\{g^k(x), 0\leq k\leq q-1\}$ such that, $$\frac{1}{k}\sum_{j=n-k}^{n-1}\overline{\lambda}_{j}^e>(1-\frac{1}{1000})a, \forall 1\leq k\leq n.$$

As a result, if $\eta$ is sufficiently small, $\theta_0$ is sufficiently closed to $1$, we have $\frac{(1-\eta)\theta-\frac{999}{1000}}{1-\frac{999}{1000}} > \frac{3}{4}$, then at least $\frac{q}{2}$ points in $\{g^k(x), 0\leq k\leq q-1\}$ are good in the orbit of $x$ as in Definition \ref{2.temp def of good point}.
\end{proof}

A key estimate due to area preservation is the following lower bound on the anglen $\angle(v_n^s,v_n^u)$ at good points.
\begin{lemma}\label{lemma 2.4}
There exist $A_0>0, \theta_0\in(0,1)$ such that if $A>A_0, \theta\in[\theta_0,1)$, $g\in {\rm Diff}^1_{\text{vol}}(S)$ satisfies $\norm{Dg}\leq A$, and $(q,a,\theta, x)$ satisfies \eqref{10x}, where $a=\log A$, then for all good points $g^n(x)$ in the orbit of $x$,  $$|\cot\angle(v_n^s,v_n^u)|\leq A^3$$
\end{lemma}
\begin{proof}
We first need a straightforward geometric lemma which only use the fact that $g$ is area-preserving.
\begin{sublemma} \label{lemma 2.5}
$|\cot\angle(v_{i+1}^u,v_{i+1}^s)|\leq e^{2\lambda_i^s}|\cot\angle(v_i^u,v_i^s)|+A^2.$
\end{sublemma}
\begin{proof}
Consider the matrix of $Dg(g^i(x))$ under the basis $(v_i^s, {v_i^s}^{\perp})$, $(v_{i+1}^s, {v_{i+1}^s}^{\perp})$, since $g$ preserves area, for some $d_i \in \R$ we have:
$$Dg(v_i^s)= e^{\lambda_i^s}v_{i+1}^s, Dg({v_i^s}^\perp)=\pm e^{-\lambda_i^s}({v_{i+1}^s}^\perp)+d_i v_{i+1}^s$$
And $|d_i|,e^{\lambda_i^s}\leq A$.
Then we have 
\begin{eqnarray*} 
|\cot\angle(v_{i+1}^u,v_{i+1}^s)|\leq e^{2\lambda_i^s}|\cot\angle(v_i^u,v_i^s)|+ |e^{\lambda_i^s}d_i|\\
\leq e^{2\lambda_i^s}|\cot\angle(v_i^u,v_i^s)|+A^2 
\end{eqnarray*}
\end{proof}

As a consequence, we have $$|\cot\angle(v_n^u,v_n^s)|\leq e^{2\sum_{j=1}^n \lambda_{n-j}^s}|\cot\angle(v_0^u,v_0^s)|+A^2(1+\sum_{j=1}^n e^{2\sum_{k=1}^j\lambda_{n-k}^s})$$

Since $g^n(x)$ is a good point in the orbit of $x$, by \eqref{bwd good} the triple $(g^{n}(x), v^s_n, v^u_n)$ is backward $(n,a)-$good. By \eqref{backward s-good} we have $$\frac{1}{j}\sum_{k=1}^j\lambda_{n-k}^s<-\frac{1}{2}a$$ for each $1\leq j \leq n$. Note that $\cot\angle(v_0^u,v_0^s)=0$, we have $$|\cot\angle(v_n^u,v_n^s)|\leq \frac{A^2}{1-e^{-a}}\leq \frac{A^2}{1-A^{-1}}\leq A^3$$
The last inequality holds when $A_0$ is large enough.
\end{proof}

Now we can conclude the proof of Proposition \ref{main prop in sec 2}. By Lemma \ref{half pts r good}, there exist $A_0, \theta_0$ such that for $(g, A, \theta, q, x)$ satisfying  \eqref{10x} and  the conditions of Proposition \ref{main prop in sec 2}, there are more than $\frac{q}{2}$ points in $\{g^k(x), 0\leq k\leq q-1\}$ are good in the orbit of $x$.

We consider the space $S\times T^1S\times T^1S$. By the pigeonhole principal, if $A_0, H_0$ are large enough (recall that $q \geq A_0^{H_0}$),then $q$ is large, and we can find two good points, $g^i(x), g^{i+L}(x)$  such that the distances $d_{T^1S}(v_i^s, v_{i+L}^s), d_{T^1S}(v_i^u, v_{i+L}^u)$ satisfy $$d_{T^1S}(v_i^s, v_{i+L}^s)<q^{-\frac{1}{100}}$$ $$d_{T^1S}(v_i^u, v_{i+L}^u)<q^{-\frac{1}{100}}$$

We claim that the point $g^i(x)$ is a $(q,a)-$good point. In fact, we consider the triple $(g^i(x), v_i^s, v_i^u)$. From our choice of $g^i(x),g^{i+L}(x)$ and $\frac{Dg^L(v_i^{s,u})}{\norm{Dg^L(v_i^{s,u})}}=v_{i+L}^{s,u}$, we know the triple $(g^i(x), v_i^s, v_i^u)$ satisfies condition (4) in Definition \ref{(q,a)-good point}. By \eqref{fwd good 2} of Definition \ref{2.temp def of good point}, since $g^i(x)$ is a good point in the orbit of $x$, we have that the triple $(g^i(x),v_i^s,v_i^u)$ is forward $(L,a)-$good. Similarly by \eqref{bwd good} in the Definition \ref{2.temp def of good point}, since  $g^{i+L}(x)$ is a good point in the orbit of $x$, we have the triple $(g^{i+L}(x), v_{i+L}^s, v_{i+L}^u)$ is backward $(L,a)-$good. Moreover, by Lemma \ref{lemma 2.4} on the estimates of the angles at points that are good in the orbit of $x$, we have $$\log|\cot\angle(v_i^s,v_i^u)|\leq 3a$$ $$\log|\cot\angle(v_{i+L}^s, v_{i+L}^u)|\leq 3a$$

As a result the triple $(g^i(x), v_i^s, v_i^u)$ satisfies conditions (1)-(4) in Definition \ref{(q,a)-good point}, so $g^i(x)$ is a $(q,a)-$good  point. The proof of Proposition \ref{main prop in sec 2} is thus complete.
\end{proof}

\subsection{From a good point to a hyperbolic periodic point}\label{construct graph transf}

We now claim that the existence of a $(q,a)-${\it good} point implies the existence of a hyperbolic periodic point when $q$ is large compared to the $C^2$ norm. 

\begin{prop} \label{main prop in sec 3}
      There exist $A_0>0$, $H_0>0$ such that the following is true.
      If $(g,A,D,q)$ satisfies: $A\geq A_0, D\geq A, q\geq D^{H_0}$ and \eqref{p1},\eqref{p2} of Theorem \ref{main prop} and if there exists a $(q,a)-$good point $x$, where $a=\log A$, then  $g$ has a hyperbolic periodic point.
\end{prop}

Using Proposition \ref{main prop in sec 3} and Proposition \ref{main prop in sec 2}, we can easily deduce Theorem  \ref{main prop}.

\begin{proof} [Proof of Theorem \ref{main prop}]
   We take $\theta_0$ as in Proposition \ref{main prop in sec 2}, $A_1$,$H_1$ as in Proposition \ref{main prop in sec 2} and $A_2$,$H_2$ as in Proposition \ref{main prop in sec 3}. Let $A_0 =\max(A_1,A_2)$,$H_0=\max(H_1,H_2)$.
   Then for $(g,q,A,D,\theta)$ satisfying the hypothesis in Theorem \ref{main prop}, we know from Proposition \ref{main prop in sec 2} that  there exists a $(q,a)-$good point $x_0 \in S$.
   From Proposition \ref{main prop in sec 3}, we know that there exists a hyperbolic periodic point.
\end{proof}

The rest of the paper is devoted to the proof of Proposition \ref{main prop in sec 3}. Hereafter, we will assume that the conditions of Proposition \ref{main prop in sec 3} are satisfied.

\subsubsection{Boxes, strips and hyperbolic-like  maps}
Let $x_0$ be the $(q,a)-$ good point in Proposition \ref{main prop in sec 3}, let $L \in \mathbb{N}$ be associated to $x_0$ as in Definition \ref{(q,a)-good point}.
To prove Proposition \ref{main prop in sec 3} we will show that $g^L$ has a hyperbolic fixed point in the neighborhood of $x_0$. We will do so by applying to a conjugate of $g^L$( by a coordinate map), denoted by $G$, a standard result on hyperbolic-like  maps, namely Proposition \ref{hyper thm} below which essentially states that if a homeomorphism $G$ of $\R^2$ maps a vertical "strip" $\cR_1$ into a horizontal "strip" $\cR_2$ that crosses $\cR_1$  transversally then $G$ has a fixed point inside the intersection of the strips. Moreover, if $G$ is a diffeomorphism and if $G$ and $G^{-1}$ have strict cone conditions in $\cR_1$ and $\cR_2$ (see \eqref{preserve horizontal cone}, \eqref{preserve vertical cone})  then the fixed point is hyperbolic. 

In this section, we first need to introduce some notations and then  define the properties of hyperbolic-like  maps that we will be interested in since they guarantee the existence of a hyperbolic fixed point. We then state Proposition \ref{the real main prop in sec 3} that reduces the proof of  Proposition \ref{main prop in sec 3}  to checking the hyperbolic-like  properties for the map $G$. 
The rest of the paper will then be dedicated to the proof of Proposition \ref{the real main prop in sec 3}.

\medskip 
\noindent $\bullet$ Given $r > 0, \tau >0, \kappa>0$, we define a {\it $(r, \tau,\kappa)-$Box }, that we denote by $U(r,\tau,\kappa)$, to be
\begin{eqnarray*}
U(r,\tau,\kappa) = \{(v,w) \in \R^2; |v| \leq r, |w| \leq \tau + \kappa |v| \}
\end{eqnarray*}
We denote $U(\infty,\tau,\kappa) = \{(v,w) \in \R^2; |w| \leq \tau + \kappa |v| \}$

\medskip 
\noindent $\bullet$ A curve contained in $\mathbb{R}^2=\mathbb{R}_x\oplus \mathbb{R}_y$ is called a {\it $\kappa-$horizontal graph} if it is the graph of a Lipschitz function from an closed interval $I\subset\mathbb{R}_x$ to $\mathbb{R}_y$ with Lipschitz constant less than $\kappa$. Similarly, we can define the {\it $\kappa-$ vertical graphs}.

\medskip 
\noindent $\bullet$ The boundary of an $(r,\tau,\kappa)-$Box $U$  is the union of two $0-$ vertical graphs and two $\kappa-$ horizontal graphs. We call these graphs respectively, the {\it left (resp. right) vertical boundary of $U$ } and the {\it upper (resp. lower) horizontal boundary of $U$}. We call the union of the left and right vertical boundary of $U$ the {\it vertical boundary of $U$}. Similarly we call the union of the upper and lower horizontal boundary of $U$ the {\it horizontal boundary of $U$}.

\medskip 
\noindent $\bullet$ Horizontal and vertical graphs that connect the boundaries of $U$ will be called full horizontal and  full vertical graphs as in the following definition. Given $r,\tau,\kappa, \eta >0$,
for each $(r,\tau, \kappa)-$Box $U$, an {\it $\eta-$ full horizontal graph of $U$} is an $\eta-$  horizontal graph $L$ such that $L \subset U$ and the endpoints of $L$ are contained in the vertical boundary of $U$. Similarly, we define  {the \it $\eta-$full vertical graphs of $U$}.

\medskip 
\noindent $\bullet$  We define an {\it $\eta-$horizontal strip of $U$} to be a subset of $U$ bounded by the vertical boundary of $U$ and two disjoint $\eta-$ full horizontal graphs of $U$ that are both disjoint from the horizontal boundary of $U$.
Similarly we can define {\it $\eta-$vertical strips of $U$}.

\medskip 
\noindent $\bullet$ A homeomorphism that maps a strip  $\mathcal{R}'$ to $\mathcal{R}$ is said to be {\it regular} if it maps  the horizontal (resp. vertical) boundary of $\mathcal{R}'$ homeomorphically to the horizontal (resp. vertical) boundary of $\mathcal{R}$.


\medskip 
\noindent $\bullet$ For $\kappa>0$, we denote
\begin{eqnarray*}C(\kappa)&=&\{(v,w)\in \mathbb{R}^2: |w|< \kappa |v|\}\\
\tilde{C}(\kappa)&=&\{(v,w)\in \mathbb{R}^2: |v|< \kappa |w|\}
\end{eqnarray*}
we will refer to these sets as cones.

Using the above definitions and notations, 
we now define a class of maps on the plane that display some sort of local hyperbolicity that  we will call {\it hyperbolic-like  maps}.

\begin{definition}
Given $r,\tau >0, 0 <\kappa,\kappa',\kappa''<1$. Denote $U =U(r,\tau,\kappa)$, and let $\mathcal{R}_1$ be a $\kappa-$vertical strip of $U$, $\mathcal{R}_2$ be a $\kappa-$horizontal strip of $U$. A diffeomorphism $G : \mathcal{R}_1 \to G(\ {\mathcal R}_1) \subset \R^2$ is called a hyperbolic-like  map if it satisfies the following conditions:
\begin{eqnarray}
\label{R1R2} \mbox{$G$ is a regular map from $\mathcal{R}_1$ to $\mathcal{R}_2$},\\
\label{preserve horizontal cone}\forall x\in \mathcal{R}_1, DG_x(C(\kappa'))\subset C(\frac{1}{2}\kappa'),\\ 
\label{preserve vertical cone}\forall x\in \mathcal{R}_2, DG^{-1}_x(\tilde{C}(\kappa")) \subset \tilde{C}(\frac{1}{2}\kappa")
\end{eqnarray}

\end{definition}

\begin{figure}[ht!]
\centering
\includegraphics[width=120mm]{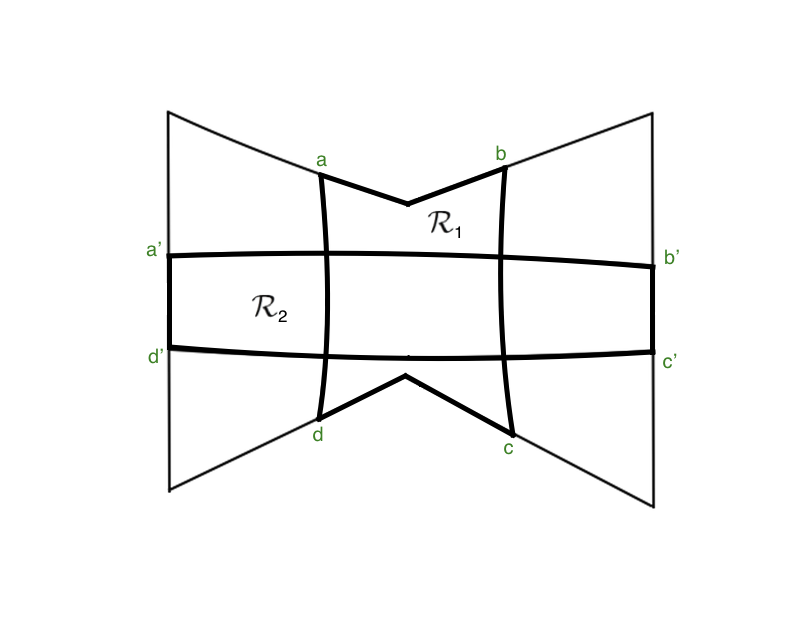}
\caption{ $\mathcal{R}_1$ is the topological rectangle $abcd$,  $\mathcal{R}_2$ is the topological rectangle $a'b'c'd'$. Under a hyperbolic-like  map $G$, $ab$ is mapped to $a'b'$. Similarly, $bc,cd,da$ are mapped respectively to $b'c',c'd',d'a'$. \label{overflow}}
\end{figure}

Our interest in hyperbolic-like  maps comes from the following classical result.
\begin{prop}\label{hyper thm}
A hyperbolic-like  map $F$ has a hyperbolic fixed point in $\mathcal{R}_1 \bigcap \mathcal{R}_2$.
\end{prop} 
\begin{proof}
By \eqref{R1R2} the Jordan curve $\mathcal C=\partial \mathcal{R}_1$ contains no fixed point of $G$. Therefore the Lefschetz index of $G$ relative to $\mathcal C$ is well defined.  From \eqref{R1R2}, we see that it is equal to $1$ therefore there is a fixed point inside $\mathcal{R}_1$ ( then necessarily in $\mathcal{R}_2$, for $G(\mathcal{R}_1) = \mathcal{R}_2$, see Figure 1). The fixed point is hyperbolic by  \eqref{preserve horizontal cone} and \eqref{preserve vertical cone}.
\end{proof}

By Proposition \ref{hyper thm}, the proof of Proposition \ref{main prop in sec 3} is thus reduced to the following proposition.
\begin{prop} \label{the real main prop in sec 3}
Under the conditions of Proposition \ref{main prop in sec 3}, there exists an integer $L > 0$, a Box $U \subset \R^2$, a vertical strip of $U$ denoted by $\mathcal{R}_1$, a $C^1$ diffeomorphism $H : U \to S$, such that $g^{L}H(\mathcal{R}_1) \subset H(U)$ and $H^{-1}g^{L}H : \mathcal{R}_1 \to \R^2$ is a hyperbolic-like  map.
\end{prop} 

Now the rest of the paper is devoted to the proof of Proposition \ref{the real main prop in sec 3}.

\subsubsection{Map from the n-th point to the (n+1)-st point}

Let $x_0$ be the $(q,a)-$ good point in Proposition \ref{main prop in sec 3}, let $v_s,v_u \in T_{x_0}S , L \in \mathbb{N}$ be as in Definition \ref{(q,a)-good point}.
From now on to the end of this paper, we denote
\begin{eqnarray*}
x_n = g^{n}(x_0), \forall 1\leq n \leq L
\end{eqnarray*}
 
To show Proposition \ref{the real main prop in sec 3}, we will study the iterates of $g$ in the neighborhoods of the piece of orbit of $x_0$ of length $L$. We want to show that the  "return" map from $x_0$ to the neighborhood of $x_0$ (since $x_L$ is very close to $x_0$) is a hyperbolic-like  map. But for this, we need to specify the Box on which the "return" map will be considered  as well as the neighborhoods of the piece of orbit of $x_0$ on which the dynamics of $g$ will be studied.

For each $0 \leq n\leq L$, let $x_n, v^{u}_n, v^{s}_n, \lambda^{s}_n, \lambda^{u}_n$ be defined in Definition \ref{2.def of stable unstable direction} and $\lambda^{e}_n$ be defined in \eqref{label lambda e} associated to the triple $(x_0, v_s, v_u)$.
In the following we fix a splitting  $T_{x_n}S=E^u_n \oplus E^s_n$, where $E^u_n=\mathbb{R}v^u_n$,$E^s_n=\mathbb{R}v^s_n$.  We define $i_n : \mathbb{R}^2 \to T_{x_n}S$ as
$$ i_n (a,b) = a v^u_n + b v^s_n $$
By \eqref{p1} and $D \geq A$, 
it is direct to see that there exists a constant $R > 0$ depending only on $S$ such that  $\exp_{x_{n+1}}^{-1}$ is  a diffeomorphism  over $g\exp_{x_n}i_n(B(0,D^{-1}R))$. 

Denote $g_n$ the $C^2$ diffeomorphism  defined by
\begin{eqnarray}
\label{def g n}  g_n: B(0,D^{-1}R) &\to& \mathbb{R}^2 \\
g_n(v,w)&=&i_{n+1}^{-1}\exp_{x_{n+1}}^{-1}g\exp_{x_n} i_n(v,w) \nonumber
\end{eqnarray}

In Section \ref{return map and the end of the proof}, we will construct the box $U$ and the vertical strip $\mathcal{R}_1$ in Proposition \ref{the real main prop in sec 3}, and we will let $H$ be $exp_{x_0}i_0$ restricted to $\cR_1$.

We let $\delta:= \frac{1}{100}a$, where  $a = \log A$. 
The following parameters $r_n,\tau_n,\kappa_n$ will determine the domains and the cones that we will consider in the study of the dynamics of $g_n$. For $0 \leq n \leq L-1$ define 
\begin{eqnarray}
c_0 : = 1 , c_{n+1} : = \min( e^{\lambda_n^e - \delta} c_n, 100 ), 
\end{eqnarray}
and 
\begin{eqnarray}
\label{def r n}r_n &=& \bar{r} c_n^3 \\
\label{def tau n}\tau_{n} &=& e^{\sum_{i=0}^{n-1} (\lambda_i^s + \delta)}\bar{r}  \\
\label{def kappa n} \kappa_{n} &=& \bar{\kappa} c_n^{-1} \\
\label{def tilde kappa n} \tilde{\kappa}_{n} &=& \bar{\kappa} c_n \\
\label{def beta n} \beta_{n} &=& \bar{\beta} c_n^{-1} \color{black} 
\end{eqnarray}
where we set
\begin{eqnarray}
\bar{r}  = D^{-3M}, \bar{\kappa}  = D^{-M}, \bar{\beta} = D^{M}
\end{eqnarray}

Here $M:=1000$.

Define for each $0\leq n\leq L$ the Box and cones 
$$U_n = U(r_n, \tau_n, \kappa_n), \quad C_n =C(\kappa_n),\quad \tilde{C}_n = \tilde{C}(\kappa_n).$$

In order to exploit the  properties of the good point $x_0$ we chose to study the dynamics of the iterates of $g$ in  the {\it moving} frames $E^u_n \oplus E^s_n$, that is $g_n$ as defined in \eqref{def g n}. This comes at the cost of an extra factor in the   $C^2$ norm of $g_n$ coming from the angle  between $v_n^s$ and $v_n^u$. The parameter $\beta_n$ gives a bound on how small can these angles become as we shall now see. 

\begin{lemma}\label{themeaningofbetan}
If we choose $A_0$ in Proposition \ref{main prop in sec 3} to be sufficiently large, we have
\begin{eqnarray}\label{beta n meaning}
\beta_i &\geq& \max(|\cot\angle(E_{i}^u,E_{i}^s)|, 1), \forall 0\leq i \leq L-1
\end{eqnarray}
\end{lemma}
\begin{proof}
Since $\delta=a/100$, if $A_0 > e^{100}$, then $\delta > 1$. By Lemma \ref{lemma 2.5}, we have $$|\cot\angle(E_{i+1}^u,E_{i+1}^s)|\leq e^{2\lambda_i^s}|\cot\angle(E_i^u,E_i^s)|+A^2$$
   Then 
   \begin{equation} \label{cotcot}
   |\cot\angle(E_{i+1}^u,E_{i+1}^s)| \leq \max (|\cot\angle(E_{i}^u,E_{i}^s)| e^{2\lambda^s_i+\delta}  ,\frac{1}{100} \bar{\beta} ) 
   \end{equation}
  since $\frac{1}{100}\bar{\beta} = \frac{1}{100}D^{M} \geq 10 A^2 $ when $A_0$ is sufficiently large.
   
   Moreover, we have
   \begin{eqnarray}
   \label{betabeta} \beta_{i+1} 
   & =&\frac{1}{\min(e^{\lambda^e_i-\delta}c_i , 100)} \bar{\beta} \\
   & \geq& \max ( e^{2\lambda^s_i + \delta} \beta_i, \frac{1}{100}\bar{\beta} ) \nonumber
   \end{eqnarray}
   The last inequality follows from $\lambda^e_i  \leq -2\lambda^s_i$. 
   Since $x_0$ is a $(q,a)-$ good point, by (3) in Definition \ref{(q,a)-good point} we have $ |\cot\angle(E_{0}^u,E_{0}^s)| < A^3 < D^{M} = \beta_0$.
   
   If for some $0\leq i\leq L-1$, we have $\beta_i \geq |\cot \angle(E_i^u,E_i^s)|$. Then by \eqref{betabeta} and \eqref{cotcot} we have
   \begin{eqnarray*}
    \beta_{i+1} &\geq&  \max ( e^{2\lambda^s_i + \delta} \beta_i, \frac{1}{100}\bar{\beta} )  \\
   & \geq& \max ( e^{2\lambda^s_i + \delta}|\cot\angle(E_{i}^u,E_{i}^s)| , \frac{1}{100}\bar{\beta} ) \\
   & \geq& |\cot\angle(E_{i+1}^u,E_{i+1}^s)|
   \end{eqnarray*}
   Thus we can show inductively that $\beta_i \geq |\cot \angle(E_i^u,E_i^s)|$ for all $0\leq i \leq L$.

Since $c_n \leq 100$, $\beta_n \geq \frac{1}{100}\bar{\beta} >1$ for all $0 \leq n \leq L$. Hence $\beta_i \geq \max(|\cot\angle(E_{i}^u,E_{i}^s)|, 1)$ for $0\leq i \leq L-1$. This completes the proof.

\end{proof}

 The following lemmata exploit the fact that $x_0$ is good to control the shape and size of the Boxes at each time $n \leq L$. Most importantly, $U_L$ is shown to be long in the horizontal direction and thin in the vertical direction compared to $U_0$.

\begin{lemma}\label{infigure4}
If we choose $A_0$ in Proposition \ref{main prop in sec 3} to be sufficiently large, then we have that $c_L = 100$. As a consequence, we have
\begin{eqnarray} 
\label{r L kappa L tilde kappa L}r_L = 10^{6} \bar{r}, \hspace{.5 cm}\kappa_L = \frac{1}{100} \bar{\kappa},\hspace{.5cm} \tilde{\kappa}_L = 100 \bar{\kappa}. 
\end{eqnarray}
\end{lemma}

\begin{proof}
Using the fact that $x_0$ is a $(q,a)-$ good point, it is easy to show that $c_L = 100$ when $A$ is large.
Indeed,
let $i$ be the biggest index such that $0 \leq i\leq L-1$ and $c_i \geq1$ ( $i\geq 0$ because $c_0=1$). 
Then $c_{L}=\min(e^{\lambda^{e}_{L-1} - \delta}c_{L-1}, 100) = \min(e^{\sum_{k=i}^{L-1}(\lambda^e_k-\delta)}c_i,100)$. Since the triple $(x_{L}, v_L^u, v_L^s)$ is  backward $(L,a)-$good, therefore by \eqref{tmp -1} $\sum_{k=i}^{L-1}(\lambda^e_k-\delta)> \frac{a}{2}$. Then $c_{L}=100$ when $A \geq 10^4$.

\end{proof}

\begin{lemma}\label{infigure42}
If we choose $A_0$ in Proposition \ref{main prop in sec 3} to be sufficiently large, we have 
\begin{eqnarray}
\label{tau L small} \tau_{L} &\leq&  \frac{1}{10}\bar{r}  \\
\label{ r n tau n} r_n &\geq&  \tau_n, \quad \forall 0 \leq n \leq L 
\end{eqnarray}
\end{lemma}

\begin{proof}
By \eqref{5x} and \eqref{def tau n}, when $A_0$ is sufficiently large
\begin{eqnarray*}
 \tau_{L} \leq A^{-\frac{1}{5}} \tau_0 \leq \frac{1}{10}\tau_0 = \frac{1}{10}\bar{r}
\end{eqnarray*}

Moreover, since $r_{n+1} = \bar{r}c_{n+1}^3 = \min(e^{3(\lambda_n^{e} - \delta)}r_n, 10^{6}\bar{r})$, then when  $\lambda_n^{e} - \delta < 0$ we have
\begin{eqnarray*}
r_{n+1} \geq e^{3(\lambda_n^{e} - \delta)} r_n
\end{eqnarray*}
and otherwise we have
\begin{eqnarray*}
r_{n+1} \geq  r_n
\end{eqnarray*}

In any case, for any $0\leq n\leq L-1$ we have $r_{n+1} \geq e^{\min(3(\lambda_n^{e} - \delta), 0)}r_n$.
Since $x_0$ is a $(q,a)-$good point,  Lemma \ref{l2.1} yields \begin{eqnarray*}
 \sum_{k=0}^{n-1}\min(3(\lambda_n^{e} - \delta), 0) \geq  \sum_{k=0}^{n-1}\min(3\lambda_n^{e}, 0)- 3\delta \geq \sum_{k=0}^{n-1} (\lambda_n^s + \delta) 
\end{eqnarray*}
Thus we have that for any $0 \leq n \leq L$, $r_n \geq  \tau_n.$ \end{proof}

\begin{rema}
Since $U_n \subset B(0,r_n + \tau_n + \kappa_n r_n)$, and
\begin{eqnarray*}
r_n + \tau_n + \kappa_n r_n \leq 2\times 10^6\bar{r} + 100\bar{\kappa} \bar{r} \leq D^{-M}
\end{eqnarray*}
when $D$ is sufficiently large.
Thus $U_n \subset B(0, D^{-1}R)$ when $D$ is sufficiently large depending only on $S$. This shows that $g_n$ is defined over $U_n$.
\end{rema}

The following proposition is purely analytic and relies on the relations between the parameters $r_n, \tau_n, \beta_n, \kappa_n, \tilde{\kappa}_n$. It will be used repeatedly in the sequel.

\begin{prop} \label{main lemma}
 Recall that $\{(r_n, \tau_n, \beta_n, \kappa_n, \tilde{\kappa}_n)\}_{0 \leq n \leq L}$ are defined by \eqref{def r n} to \eqref{def beta n}.  Assume that we also have \eqref{beta n meaning} and \eqref{ r n tau n}. Then for all $D$ sufficiently large depending only on $S$, for each $0 \leq n \leq L-1$ we have the following:

(1). For any $(v,w) \in U_n$, we have $(Dg_n)_{(v,w)}(C_n) \subset C_{n+1}$; For any $(v,w) \in g_n(U_n) \bigcap U_{n+1}$, we have $(Dg_n^{-1})_{(v,w)}(\tilde{C}_{n+1}) \subset \tilde{C}_{n}$.  As a consequence, the image of any $\kappa_n-$ horizontal graph contained in $U_n$ under $g_n$ is a $\kappa_{n+1}-$ horizontal graph; the image of any $\tilde{\kappa}_{n+1}-$ vertical graph contained in $g_n(U_n) \bigcap U_{n+1}$ under $g_n^{-1}$ is a $\tilde{\kappa}_{n}-$ vertical graph.

(2). If $\Gamma$ is a $\kappa_n-$full horizontal graph of $U_n$, then $g_n(\Gamma) \bigcap U_{n+1}$ is a $\kappa_{n+1}$-full horizontal graph of $U_{n+1}$. Moreover, the image of the horizontal boundary of $U_n$ under $g_n$ is disjoint from the horizontal boundary of $U_{n+1}$; the image of the vertical boundary of $U_n$ under $g_n$ is disjoint from the vertical boundary of $U_{n+1}$.

\end{prop}

The proof of Proposition \ref{main lemma} will be given in the Appendix.

\begin{cor} \label{for n = 1}
For any $0 \leq k \leq L-1$,  we can find a $\tilde{\kappa}_k-$ vertical strip of $U_k$, denoted by $\mathcal{R}'$, such that $\mathcal{R} =g_{k}(\mathcal{R'})$ is a $\kappa_{k+1} -$ horizontal strip of $U_{k+1}$. Moreover, $g_k$ is a regular map from $\mathcal{R}'$ to $\mathcal{R}$.
\end{cor}

\begin{figure}[ht!]
\centering
\includegraphics[width=120mm]{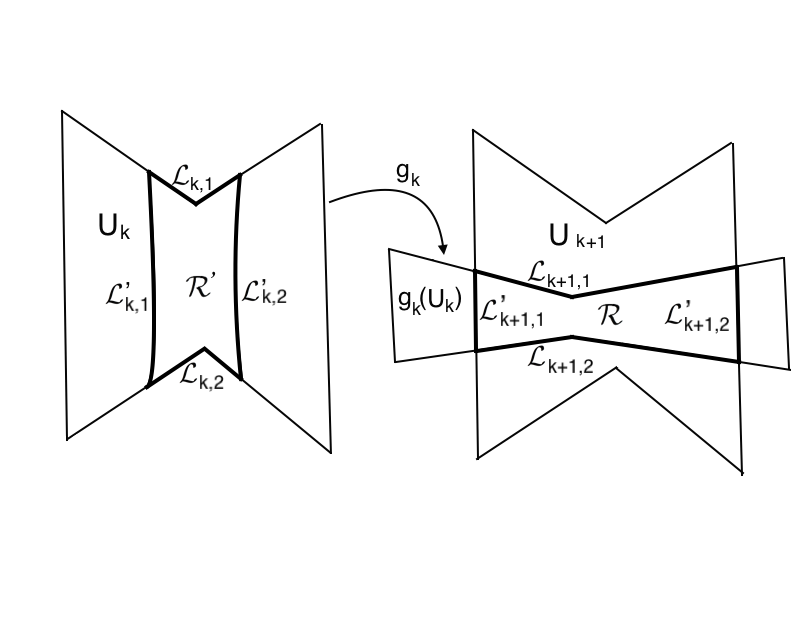}
\caption{ $g_k$ maps $\cR'$ to $\cR$. Moreover $\cL_{k,1}, \cL_{k,2}$ are mapped to $\cL_{k+1,1},\cL_{k+1,2}$ respectively and $\cL'_{k,1}, \cL'_{k,2}$ are mapped to $\cL'_{k+1,1},\cL'_{k+1,2}$ respectively. }
\end{figure}

\begin{proof}

For any $0 \leq k \leq L-1$, denote the upper and lower horizontal boundary of $U_k$ by $\mathcal{L}_1, \mathcal{L}_2$ respectively. Then by Proposition \ref{main lemma}, for $i=1,2$, there exists a $\kappa_k-$ horizontal graph contained in $\mathcal{L}_i$, denoted by $\mathcal{L}_{k,i}$, such that $g_k(\mathcal{L}_{k,i})$ is a $\kappa_{k+1}-$ full horizontal graph of $U_{k+1}$, denoted by $\mathcal{L}_{k+1,i}$. By Proposition \ref{main lemma}, $\mathcal{L}_{k+1,1}, \mathcal{L}_{k+1,2}$ are disjoint from the horizontal boundary of $U_{k+1}$. Then there exist a horizontal strip of $U_{k+1}$, denoted by $\mathcal{R}$,  bounded by $\mathcal{L}_{k+1,1}$ and $\mathcal{L}_{k+1,2}$ ( see Figure 2 ).

Again by Proposition \ref{main lemma}, denote the $\tilde{\kappa}_{k+1}-$vertical graph contained in the left vertical boundary of $U_{k+1}$, connecting the endpoints of $\mathcal{L}_{k+1,1}$ and $\mathcal{L}_{k+1,2}$ by $\mathcal{L}'_{k+1, 1}$, then $g_k^{-1}(\mathcal{L}'_{k+1,1})$ is a $\tilde{\kappa}_{k}-$vertical graph, denoted by $\mathcal{L}'_{k,1}$. From Proposition \ref{main lemma}, we see that the images of the vertical boundary of $U_k$ under $g_k$ is disjoint from the vertical boundary of $U_{k+1}$. Then $\mathcal{L}'_{k,1}$ is disjoint from the vertical boundary of $U_k$. Since by construction $\mathcal{L}'_{k,1}$ connect the upper and lower horizontal boundary of $U_k$, we have $\mathcal{L}_{k,1}' \subset U_k$. Hence $\mathcal{L}_{k,1}' $ is a $\tilde{\kappa}_{k}-$full vertical graph of $U_k$.  In a similar fashion, we construct  $\mathcal{L}'_{k+1,2}$ and $\mathcal{L}'_{k,2}$. Then there exists a vertical strip of $U_k$, denoted by $\mathcal{R}'$, bounded by $\mathcal{L}'_{k,1}$ and $\mathcal{L}'_{k,2}$. It is straightforward to check that the statement is true for $\mathcal{R}'$ and $\mathcal{R}$. 
\end{proof}


\subsubsection{Concatenation}

In this section, we will use Corollary \ref{for n = 1}    inductively to   construct a vertical strip of $U_0$ that is mapped by $g_{L-1}\cdots g_0$ to a horizontal strip of $U_L$. We will also prove the cone preservation  by the derivative map using Proposition \ref{main lemma}. The result of the induction is summarized in Corollary \ref{cor summary}.

First we need the following lemma.
\begin{lemma} \label{unique intersection}
If $D>1$, for any $0 \leq n \leq L$, any $\tilde{\kappa}_n-$vertical full graph $\cL'$, any $\kappa_n-$horizontal full graph $\cL$, $\cL'$ intersects $L$ at a unique point.
\end{lemma}
\begin{proof}
A simple topological argument shows that there is at least one intersection point. 

If there are two intersection points, then by the Rolle's theorem, there exists $p \in \cL'$ and $q \in \cL$, such that the Lipschitz constant of $\cL$ at $q$, denoted by $t$, equals to the inverse of the Lipschitz constant of $\cL'$at $p$, denoted by $t'^{-1}$. We have $\bar{\kappa} = D^{-M} < 1$. Thus, we have $ \tilde{\kappa}_n \geq t' = t^{-1} \geq \kappa_n ^{-1} = c_n \bar{\kappa}^{-1} > c_n \bar{\kappa} = \tilde{\kappa}_n$. This is a contradiction.
\end{proof}


\begin{cor}\label{cor summary}
When $A_0 > 0$ is sufficiently large,  there exists a $\tilde{\kappa}_0-$ vertical strip of $U_0$, denoted by $\mathcal{R}'$, such that $\mathcal{R} =g_{L-1}\cdots g_0(\mathcal{R}')$ is a $\kappa_L-$ horizontal strip of $U_{L}$ and $g_{L-1}\cdots g_0$ is a regular map from $\mathcal{R}'$ to $\mathcal{R}$. Moreover, for any $x\in \mathcal{R}'$, $D(g_{L-1}\cdots g_0)_{x}$ maps $C(\bar{\kappa})$ into $C(\frac{1}{100}\bar{\kappa})$; For any $x \in \mathcal{R}$, $D(g_{0}^{-1}\cdots g_{L-1}^{-1})_{x}$ maps $\tilde{C}(100\bar{\kappa})$ into $\tilde{C}(\bar{\kappa})$. 
\end{cor}

\begin{proof}
We will inductively prove a stronger statement : for any $1\leq n\leq L$, there exists a $\tilde{\kappa}_0-$ vertical strip $R'_0$ of $U_{0}$ and  a $\kappa_{n}-$ horizontal strip $R_{n}$ of $U_{n}$ such that 
 $g_{n-1}\cdots g_0$ is a regular map from $R'_0$ to $R_n$.

First, we see that the above statement for $n=1$ follows from Corollary \ref{for n = 1}.

Now we assume that the statement is true for $n-1$.
By the induction hypothesis, we get a $\tilde{\kappa}_0-$ vertical strip of $U_0$, denoted by $Q'_0$, such that $g_{n-2}\cdots g_0(Q'_0)$ is a $\kappa_{n-1}-$ horizontal strip of $U_{n-1}$, denoted by $Q_{n-1}$. From Corollary \ref{for n = 1}, we get a $\tilde{\kappa}_{n-1}-$ vertical strip of $U_{n-1}$, denoted by $Q'_{n-1}$, such that $g_{n-1}(Q'_{n-1})$ is a $\kappa_{n}-$ horizontal strip of $U_{n}$, denoted by $Q_{n}$. Denote two horizontal boundaries of $Q_{n-1}$ by $M_{1}, M_2$, two vertical boundaries of $Q'_{n-1}$ by $M'_1, M'_2$ ( see Figure 3).

\begin{figure}[ht!]
\centering
\includegraphics[width=120mm]{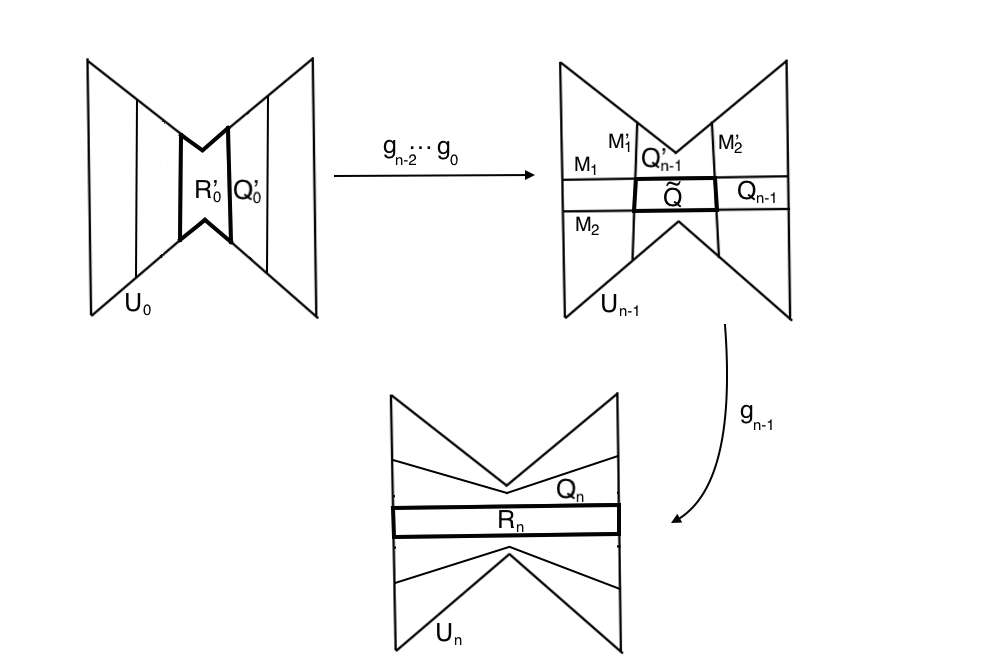}

\caption{ $g_{n-2}\cdots g_0$ maps $R_0', Q_0'$ to $\tilde{Q}, Q_{n-1}$ respectively; $g_{n-1}$ maps $\tilde{Q}, Q_{n-1}'$ to $R_{n}, Q_n$ respectively.}
\end{figure}

By Lemma \ref{unique intersection}, we easily see that $\tilde{Q} = Q_{n-1} \bigcap Q'_{n-1}$ is bounded by: two $\tilde{\kappa}_{n-1}-$ vertical graphs contained in $M'_1$ and $M'_2$ respectively, and two $\kappa_{n-1}-$ horizontal graphs contained in $M_1$ and $M_2$ respectively. Again by Proposition
 \ref{main lemma}, we have that $g_{0}^{-1}\cdots g_{n-2}^{-1}(\tilde{Q})$ is a $\tilde{\kappa}_{0}-$ vertical strip of $U_0$,denoted by $R'_{0}$, and $g_{n-1}(\tilde{Q})$ is a $\kappa_{n}-$ horizontal strip of $U_{n}$, denoted by $R_{n}$. It is clear from the definitions that $g_{n-1}\cdots g_0$ is a regular maps from $R'_{0}$ to $R_{n}$. This completes the induction.
 
 The first part of Corollary \ref{cor summary} is thus proved if we take $n=L$, and $\cR'= R'_{0}$ and $\cR=R_L$. The construction also shows that for any $x \in \mathcal{R}'$, any $0 \leq n \leq L-1$, $g_n \cdots g_0(x) \in U_{n+1}$. This allows us to apply (1) in Proposition \ref{main lemma} and get from  \eqref{r L kappa L tilde kappa L}  the second part of the statement of Corollary \ref{cor summary}.  \end{proof}

\subsubsection{Return map and the end of the proof of Proposition \ref{the real main prop in sec 3}.} \label{return map and the end of the proof}

Since we have to consider a map from a vertical strip of a box to a horizontal strip of the {\it{same}} box. We have to compare the exponential coordinate charts at $x_0$ and $x_L$.
Now we define the return map $I: U_L\to \Omega_0$ as $I:=i_0^{-1}\exp^{-1}_{x_0}\exp_{x_L}i_L$. 

\begin{lemma}\label{return lemma}If $A_0, H_0$ are sufficiently large, the map $I$ satisfies the following inequalities:
\begin{eqnarray} 
\label{return 2}\forall x\in U_L, DI_x(C(\frac{1}{100}\bar{\kappa}))\subset C(\frac{1}{2}\bar{\kappa})\\
\label{return 3}\forall x\in I(U_L), DI^{-1}_x(\tilde{C}(2\bar{\kappa}))\subset \tilde{C}(100\bar{\kappa})
\end{eqnarray}
Moreover, the image of any $\frac{1}{100}\bar{\kappa}-$ full horizontal graph of $U_L$ under $I$ contains a $\frac{1}{2}\bar{\kappa}-$ full horizontal graph of $U_0$. The image of the vertical boundary of $U_L$ under $I$ is disjoint from $U_0$.
\end{lemma}

\begin{figure}[ht!]
\centering
\includegraphics[width=120mm]{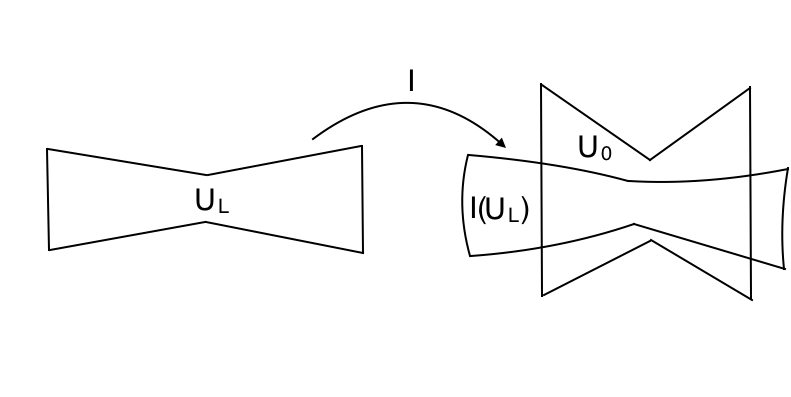}

 \caption{By Lemma \ref{infigure4},\ref{infigure42}, $U_L \subset U(10^6\bar{r}, \frac{1}{10}\bar{r},\frac{1}{100}\bar{\kappa})$, $U_0 =  U(\bar{r}, \bar{r}, \bar{\kappa})$. The image of $U_L$ under $I$ intersects $U_0$ transversely. }
\end{figure}

\begin{proof}
By (3) and (4) in the definition of a $(q,a)-$ good point, we have
\begin{eqnarray*}
d(x_0, x_L),d_{T^{1}S}(v^s_0,v^s_L),d_{T^{1}S}(v^u_0,v^u_L) < q^{-\frac{1}{100}} \leq D^{-\frac{H_0}{100}}\\
|\cot \angle(v^s_0,v^u_0)|,|\cot \angle(v^s_L,v^u_L)| \leq A^{3} \leq D^{3}
\end{eqnarray*} which shows that  the map $i_0^{-1}exp_{x_0}^{-1}exp_{x_L}i_L$ is $D^{-O(H_0)}$ close to the identity in $C^1$ over a $D^{-1}$ neighborhood of the origin when $D$ is sufficiently large ( Note that the $C^2$ norm of the exponential map $x_0$,$x_L$ depends only on $S$. Restricted to a $D^{-1}$ neighbourhood of the origin, $exp_{x_0}^{-1} exp_{x_L}$ tends to identity in $C^1$ as $d(x_0, x_L)$ tends to zero and $D$ tends to infinity). By letting $H_0$ to be larger than some constant times $M$
we obtain 
\begin{equation} \label{return 0} \norm{I^{\pm1} -id}_{C^1}\leq D^{-10M} \end{equation}
 We get \eqref{return 2},\eqref{return 3} as simple consequences of \eqref{return 0}.

Denote the left and right vertical boundary of $U_L$ by $L_1,L_2$ respectively.
By \eqref{return 0} and \eqref{r L kappa L tilde kappa L}, we see that 
\begin{eqnarray}
\label{left end left}
\sup_{(v,w) \in L_1} \pi_1(I(v,w)) &\leq& \sup_{(v,w) \in L_1} v  + D^{-10M} \\
&=& -10^{6}\bar{r} + D^{-10M}\leq -10\bar{r}   \nonumber
\end{eqnarray}
 and similarly 
 \begin{eqnarray}
\inf_{(v,w) \in L_2} \pi_1(I(v,w)) \geq 10\bar{r} \label{right end right}
 \end{eqnarray}
 here for any $(a,b) \in \R^2$, we denote $\pi_1(a,b) := a$.
Since $U_0 = U(\bar{r}, \bar{r},\bar{\kappa})$,
this proves the last statement in the lemma.

Take any $\frac{1}{100}\bar{\kappa}-$ full horizontal graph of $U_L$, denoted by $\mathcal{L}$. 
By \eqref{tau L small},
$\mathcal{L}$ is contained in $U(10^{6}\bar{r}, \frac{1}{10}\bar{r}, \frac{1}{100}\bar{\kappa} )$. By \eqref{return 0} and \eqref{return 2}, $I(\mathcal{L})$ is contained in $U(2\times 10^{6}\bar{r}, \frac{1}{2}\bar{r}, \frac{1}{2}\bar{\kappa})$. By \eqref{return 2}, \eqref{return 3}, $I(\mathcal{L})$ is a $\frac{1}{2}\bar{\kappa}-$horizontal graph. Using \eqref{left end left} and \eqref{right end right},  we see that $I(\mathcal{L})$ contains a $\frac{1}{2}\bar{\kappa}-$full horizontal graph of $U_0$. This completes the proof.
\end{proof}

We are finally ready to prove Proposition \ref{the real main prop in sec 3}.
\begin{proof}[Proof of Proposition \ref{the real main prop in sec 3}] 
Consider the map $G:=IG'$ with $G':=g_{L-1}\circ\cdots\circ g_0$.
 We have \begin{eqnarray*} G = i_0^{-1}exp_{x_0}^{-1}g^{L}exp_{x_0}i_0 \end{eqnarray*} 
When $D > 10$, we have $2\bar{\kappa} < 1$. If we show that there exists a vertical strip $\cR_1$ of $U_0$ and a horizontal strip $\cR_2$ of $U_0$ such that 
\begin{enumerate}
\item \mbox{$G$ is a regular map from $\mathcal{R}_1$ to $\mathcal{R}_2$},
\item $\displaystyle{\text{ for any }  (v,w) \in \mathcal{R}_1,  
 DG_{(v,w)}(C(\bar{\kappa})) \subset C(\frac{1}{2}\bar{\kappa})} $
\item $ \displaystyle{ \text{ for any }  (v,w) \in \mathcal{R}_2,   DG_{(v,w)}^{-1}(\tilde{C}(2\bar{\kappa})) \subset \tilde{C}(\bar{\kappa})}$
\end{enumerate}
 
Then we define $U =U_0$ and $H : \cR_1 \to S$ by $H= exp_{x_0}i_0$, we have $(H^{-1}g^{L}H)|_{\cR_1} = G$. This finishes the proof.

\medskip 

\noindent {\it Proof of (1)}. By Corollary \ref{cor summary}, we know that there is a $\tilde{\kappa}_0-$vertical strip $\mathcal{R}'$ of $U_0$ and $\mathcal{R}:=G'(\mathcal{R}')$ is a horizontal strip of $U_L$.
Denote $\mathcal{R}_2:=I(\mathcal{R})\cap U_0$. 

By Corollary \ref{cor summary},  $\mathcal{R}$ is a $\kappa_L-$horizontal strip of $U_{L}$. By Lemma \ref{return lemma} \eqref{return 2},
the image of the upper horizontal boundary of $\mathcal{R}$ under $I$ is a $2\kappa_L-$horizontal graph, and by Lemma \ref{return lemma}, it intersects the vertical boundary of $U_0$ at two points, one in each component. We have the same for the lower horizontal boundary of $\mathcal{R}$.
Then $\mathcal{R}_2$ is a $2\kappa_L-$horizontal strip of $U_0$ and the vertical boundary of $R_2$ is contained in the vertical boundary of $U_0$. 

By Corollary \ref{cor summary} and Lemma \ref{return lemma}, $(DG^{-1})_{(v,w)}(\tilde{C}(2\bar{\kappa})) \subset \tilde{C}(\bar{\kappa})$ for any $(v,w) \in \mathcal{R}_2$, this implies that the image of the vertical boundary of $R_2$ under $G^{-1}$ is the union of two $\bar{\kappa}-$ vertical graphs connecting two components of the horizontal boundary of $\mathcal{R}'$. By Lemma \ref{return lemma}, these two vertical graphs are contained in $\mathcal{R}'$, disjoint from each other and disjoint from the vertical boundary of $\mathcal{R}'$, hence are $\bar{\kappa}-$ full vertical graphs of $\mathcal{R}'$. Hence $G^{-1}(\mathcal{R}_2)$ is contained in $\mathcal{R}'$, and is bounded by two $\bar{\kappa}-$ full vertical graphs of $\cR'$. Then $G^{-1}(\mathcal{R}_2)$ is a vertical strip of $U_0$, denoted by $\mathcal{R}_1$. From the construction, we see that the horizontal (resp. vertical) boundary of $G^{-1}(\mathcal{R}_2)$ is mapped homeomorphically to the horizontal (resp. vertical) boundary of $\mathcal{R}$. Hence $G$ is a regular map between these two strips. The proof of (1) is thus finished. $\hfill \Box$

Combining Lemma \ref{return lemma} and Corollary \ref{cor summary}, we immediately have (2) and (3) , which ends the proof of Proposition \ref{the real main prop in sec 3}.  \end{proof}

\section{Appendix}

\begin{proof} [Proof of Proposition \ref{main lemma}] \label{3.proof of main lemma}

We denote the first and second coordinates of $g_n$ as $\bar{v},\bar{w}$.
We expand $g_n$ into linear terms and higher order terms at the origin. For each $(v,w) \in U_n$, we have
\begin{eqnarray} 
\label{29} \bar{v}(v,w)&=&A_nv+f_n(v,w)\\
\label{30} \bar{w}(v,w)&=&B_nw+h_n(v,w)
\end{eqnarray}

where $\partial_{v}f_n(0,0)=\partial_{w}f_n(0,0)=\partial_{v}h_n(0,0)=\partial_{w}h_n(0,0)=0$. ($f_n,h_n$ are uniquely determined by these conditions). Following the notations in Section \ref{finding orbit}, we have
\begin{eqnarray*}
A_n = e^{\lambda^u_n} ,   B_n =e^{\lambda^s_n} 
\end{eqnarray*}

 Recall that we have assumed $\norm{Dg} \leq A$ and $\norm{D^2g} \leq D$.
By Lemma \ref{themeaningofbetan}, we have the following lemmata which are essentially proved in \cite{CP}.
    
    \begin{lemma} \label{part f 1}
    For all $0 \leq n \leq L-1$ we have
    \begin{eqnarray*}
      \norm{D^2f_n} , \norm{D^2h_n}  \leq C_0D\beta_{n+1}
    \end{eqnarray*}
    As a consequence,  for all $0 \leq n \leq L-1$ we have
\begin{eqnarray*}
|\partial_vf_n(v,w)|,|\partial_wf_n(v,w)|,|\partial_vh_n(v,w)|,|\partial_wh_n(v,w)| \leq  C_0D\beta_{n+1} (|v| + |w|) 
\end{eqnarray*}
    Here $C_0$ is a constant depending only on $S$ and the norm is taken restricted to $U_n$.
    \end{lemma}

Denote $\epsilon_n=2C_0D\beta_{n+1}r_n(1+\kappa_n)$, then we have the following estimate:

\begin{lemma} \label{part f 2}
For each $0\leq n \leq L-1$, we have 
\begin{eqnarray*}  
\norm{\partial_v f_n}_{C^0} \leq \epsilon_n , \norm{\partial_wf_n}_{C^0} \leq \epsilon_n \\
\norm{\partial_v h_n}_{C^0} \leq \epsilon_n , \norm{\partial_w h_n}_{C^0} \leq \epsilon_n
\end{eqnarray*} 
Here $C^{0}$ norm is taken restricted to $U_n$.
\end{lemma}

It is easy to see that
\begin{eqnarray*}
 \epsilon_n  = 2C_0D\beta_{n+1}r_n(1+\kappa_n) &\leq& 2C_0D\max(A^{2}c_n^{-1}, \frac{1}{100}) \bar{\beta} \bar{r} c_n^3(1+c_n^{-1}\bar{\kappa}) \\
&\leq& 400C_0c_nDA^2\bar{\beta} \bar{r} 
\end{eqnarray*}
Then by \eqref{def kappa n} and \eqref{def tilde kappa n}, we have 
\begin{eqnarray}
\label{epsilon} \epsilon_n&\leq& D^{-\frac{M}{2}}\bar{\kappa} \\
\label{epsilon 2} \epsilon_n \kappa_n &\leq& D^{-M}\bar{\kappa}  \\
\label{epsilon 4} \epsilon_n \tilde{\kappa}_{n} &\leq& D^{-M}\bar{\kappa}
\end{eqnarray}
when $D$ is sufficiently large depending only on $S$.

Then by Lemma \ref{part f 1} and \eqref{29}, for any $(v,w) \in U_n$ we have,
\begin{eqnarray} \label{bar v}
|\bar{v}| &\geq& A_n|v| - |f_n(v,w)| \\
&\geq& A_n|v| - \epsilon_n(|v| + |w|) \nonumber \\
&\geq& ( A_n - \epsilon_n - \epsilon_n \kappa_n )|v| - \epsilon_n \tau_n \nonumber
\end{eqnarray}
and by \eqref{30}
\begin{eqnarray}\label{bar w}
|\bar{w}| &\leq& B_n|w| + |h_n(v,w)| \\
&\leq& B_n|w| + \epsilon_n(|v| + |w|) \nonumber \\
&\leq& B_n(\tau_n + \kappa_n |v|) + \epsilon_n((1+\kappa_n)|v| + \tau_n) \nonumber \\
&\leq& (B_n \kappa_n + \epsilon_n + \epsilon_n \kappa_n) |v| + B_n \tau_n + \epsilon_n \tau_n  \nonumber
\end{eqnarray}
We will first show that :
\begin{eqnarray}
\label{another label used only once}g_n(U_n) \mbox{ is contained in the interior of } U(\infty, \tau_{n+1}, \kappa_{n+1})
\end{eqnarray}
Combining \eqref{bar v} and \eqref{bar w}, \eqref{another label used only once} follows from :
\begin{eqnarray}
\label{rho n+1}\kappa_{n+1}&>& \frac{e^{\lambda_n^s} \kappa_n + \epsilon_n + \epsilon_n \kappa_n}{ e^{\lambda_n^u} - \epsilon_n - \epsilon_n \kappa_n } \\
\label{tau n+1}\tau_{n+1} &>& ( e^{\lambda_n^s} + \epsilon_n + \epsilon_n \frac{e^{\lambda_n^s} \kappa_n + \epsilon_n + \epsilon_n \kappa_n}{ e^{\lambda_n^u} - \epsilon_n - \epsilon_n \kappa_n } ) \tau_n 
\end{eqnarray} 
By \eqref{epsilon}, \eqref{epsilon 2}, then when $D$ is sufficiently large we get $\eqref{rho n+1}$ from
\begin{eqnarray*}
 \frac{e^{\lambda_n^s} \kappa_n + \epsilon_n + \epsilon_n \kappa_n}{ e^{\lambda_n^u} - \epsilon_n - \epsilon_n \kappa_n }  &<& e^{-\lambda_n^u + \frac{1}{2}\delta}(e^{\lambda_n^s} \kappa_n + D^{-\frac{1}{4}M}\bar{\kappa})   \\
&\leq& \max(e^{ - \lambda_n^u + \lambda_n^s + \delta} \kappa_n, \frac{1}{100}\bar{\kappa})   \\
&\leq& \kappa_{n+1}
\end{eqnarray*}

Similarly, by $\eqref{epsilon}$,$\eqref{epsilon 2}$, $\eqref{def tau n}$, we have $\eqref{tau n+1}$ when $D$ is sufficiently large.

Now we get to the proof of $(1)$ in Proposition \ref{main lemma}. Differentiate \eqref{29} and \eqref{30}, we get

\begin{eqnarray}
\label{used once}\frac{d\bar{v}}{dv} &=& A_n + \partial_v f_n(v,w)  \\
\frac{d\bar{v}}{dw} &=& \partial_w f_n(v,w) \\
\frac{d\bar{w}}{dv} &=& \partial_v h_n(v,w) \\
\frac{d\bar{w}}{dw} &=& B_n + \partial_w h_n(v,w)
\end{eqnarray}

Take any $(a,b) \in C_n$, $(v,w) \in U_n$, then
\begin{eqnarray*}
( Dg_{n} )_{(v,w)}(a,b) &=&  (\bar{a}, \bar{b}) 
\end{eqnarray*}
with
\begin{eqnarray*}
\bar{a} &=&  (A_n + \partial_v f_n(v,w) ) a +  \partial_w f_n(v,w)  b \\
\bar{b} &=&  \partial_v h_n(v,w) a + ( B_n + \partial_w h_n(v,w))b 
\end{eqnarray*}
Since
$|b| \leq \kappa_n|a|$, by Lemma \ref{part f 2} we have
\begin{eqnarray} \label{bar a norm}
|\bar{a} |&\geq& (e^{\lambda_n^u } - \epsilon_n) |a| - \epsilon_n |b|  \\
&\geq& (e^{\lambda_n^u } - \epsilon_n - \epsilon_n\kappa_n) |a|  \nonumber
\end{eqnarray}
and
\begin{eqnarray} \label{bar b norm}
|\bar{b}| &\leq& \epsilon_n a + (e^{\lambda_n^s} + \epsilon_n)b \\
&\leq&(\epsilon_n \kappa_n + e^{\lambda_n^s} \kappa_n + \epsilon_n) |a| \nonumber
\end{eqnarray}

Then we have
\begin{eqnarray*}
|\bar{b}| \leq |\bar{a}|\frac{\epsilon_n \kappa_n + e^{\lambda_n^s} \kappa_n + \epsilon_n}{e^{\lambda_n^u } - \epsilon_n - \epsilon_n\kappa_n }
\end{eqnarray*}

Then by \eqref{rho n+1}, we have that $(\bar{a}, \bar{b}) \in C_{n+1}$.

On the other hand, take any $ (\bar{a},\bar{b}) \in \tilde{C}_{n+1}, (v,w) \in g_n(U_n) \bigcap U_{n+1}$, denote $(a,b) = (Dg_n^{-1})_{(v,w)}(\bar{a}, \bar{b})$, by \eqref{bar a norm} and \eqref{bar b norm} we have
\begin{eqnarray*}
 \tilde{\kappa}_{n+1}(\epsilon_n |a| + (e^{\lambda_n^s} + \epsilon_n)|b|) \geq \tilde{\kappa}_{n+1}|\bar{b}| \geq |\bar{a}|  \geq ( e^{\lambda_n^u } - \epsilon_n) |a| - \epsilon_n |b|
\end{eqnarray*}
hence
\begin{eqnarray*}
  |a| &\leq& |b|\frac{\tilde{\kappa}_{n+1}(e^{\lambda_n^s}+ \epsilon_n) + \epsilon_n}{e^{\lambda_n^u} - \epsilon_n - \tilde{\kappa}_{n+1}\epsilon_n} 
\end{eqnarray*}
In order to prove that $(a,b) \in \tilde{C}_n$, it is enough to verify that
\begin{eqnarray*}
\frac{\tilde{\kappa}_{n+1}(e^{\lambda_n^s}+ \epsilon_n) + \epsilon_n}{e^{\lambda_n^u} - \epsilon_n - \tilde{\kappa}_{n+1}\epsilon_n}  \leq \tilde{\kappa}_n
\end{eqnarray*}
that is 
\begin{eqnarray*}
\tilde{\kappa}_{n+1} \leq \frac{\tilde{\kappa}_n e^{\lambda_n^u }- \epsilon_n \tilde{\kappa}_n - \epsilon_n}{e^{\lambda_n^s} + \epsilon_n + \epsilon_n \tilde{\kappa}_n}
\end{eqnarray*}
when $D$ is sufficiently large.
This is proved in a similar way as \eqref{rho n+1} using \eqref{epsilon}, \eqref{epsilon 4}.
This completes the proof of $(1)$. 

Now we prove $(2)$ in Proposition \ref{main lemma}.
If $\Gamma$ is $\kappa_n-$full horizontal graph of $U_n$, we can denote $\Gamma$ as the graph of $\phi : [-r_n, r_n] \to \R$, such that:

(1) The Lipschitz constant of $\phi$ is smaller than $\kappa_n$ everywhere;

(2) For any $v \in [-r_n, r_n]$, we have $(v,\phi(v)) \in U_n$.

By \eqref{used once}, we have for any $(v,w) \in U_{n}$ 
\begin{eqnarray*}
\frac{d\bar{v}}{dv}(v,w) &>& e^{\lambda_n^u} - \epsilon_n(|v| + |w|) \\
&\geq& e^{\lambda_n^u} - \epsilon_n(|v| + \kappa_n|v| + \tau_n)  \\
&\geq& e^{\lambda_n^u} - \epsilon_n(r_n + \kappa_nr_n + \tau_n)   \\
&>& 0
\end{eqnarray*}
when $D$ is sufficiently large.
By \eqref{bar v}, for any $w$ such that $(r_n, w) \in U_n$, if $D$ is sufficiently large then
\begin{eqnarray*}
\bar{v}(r_n, w) &\geq& (A_n - \epsilon_n - \epsilon_n \kappa_n ) r_n - \epsilon_n \tau_n\\
& > & e^{\lambda_n^u - \frac{1}{2}\delta } r_n > r_{n+1}
\end{eqnarray*}
Here the second inequality follows from \eqref{epsilon}, \eqref{epsilon 2} and \eqref{ r n tau n}.
Similar calculation shows that $\bar{v}(-r_n, w) < -r_{n+1}$ for all $w$ such that $(-r_n, w) \in U_n$. Then $\bar{v}^{-1}$ is defined over $[-r_{n+1}, r_{n+1}]$, and we have that the image of the vertical boundary of $U_n$ under $g_n$ is disjoint from the vertical boundary of $U_{n+1}$.

Then by \eqref{another label used only once}, we have
\begin{eqnarray*}
&&g_n(\Gamma) \bigcap U_{n+1} \mbox{ is contained in } U_n \\
&&\mbox{ and disjoint from the horizontal boundary of } U_n
\end{eqnarray*}
Moreover, $g_n(\Gamma) \bigcap U_{n+1}$ is the graph of function $\phi \bar{v}^{-1}$ restricted to $[-r_{n+1}, r_{n+1}]$. 
Since we already showed (1), we know that $g_n(\Gamma) \bigcap U_{n+1}$ is a $\kappa_{n+1}-$full horizontal graph, and the image of the horizontal boundary of $U_n$ under $g_n$ is disjoint from the horizontal boundary of $U_{n+1}$.

 This completes the proof of $(2)$.
\end{proof}


\begin{thebibliography}{aaaa}

\bibitem{AK}   D. ~V. Anosov and A. ~B. Katok,
\newblock {\it New examples in smooth ergodic theory. Ergodic diffeomorphisms,} 
\newblock {Trans. of  Moscow Math. Soc.} 23, 1--35, 1970.

\bibitem{bramham} B. Bramham, {\it Pseudo-rotations with sufficiently Liouville rotation number are $C^0$-rigid.,}  Invent. Math., posted May 21, (2014),  arXiv:1205.6243

\bibitem{CP} V. Climenhaga, Y. Pesin  {\it Hadamard-Perron theorems and effective hyperbolicity}, arXiv:1303.2375 

\bibitem{crovisier}   S. Crovisier, {\it Exotic rotations,} M\'ethodes topologiques en dynamique des surfaces, Grenoble (juin 2006), http://www.math.u-psud.fr/~crovisie/grenoble2006.pdf


\bibitem{franks.etds}  Franks, J., {\it Recurrence and fixed points of surface homeomorphisms},   Ergodic Theory Dynam. Systems {\bf 8} (1988) p. 99--107
\bibitem{franks}  Franks, J., {\it Generalizations of the Poincar\'e-Birkhoff Theorem}. Ann. Math. {\bf 128}, p. 139-151 
(1988) Erratum in Ann. of Math.  {\bf 164} (2006) p. 1097--1098.
\bibitem{franks-handel} J. Franks, M. Handel, {\it Entropy zero areas preserving diffeomorphisms of ${\mathbb S}^2$,} Geom. Topol. 16 (2012), no. 4, 2187--2284.

\bibitem{FaKa} B. Fayad and A. Katok, {\it Constructions in elliptic dynamics,}  Ergodic Theory and Dynamical Systems, (Herman memorial issue) {\bf 24}, (2004), 1477- 1520.
\bibitem{FaKa-analytic} B. Fayad and A. Katok, {\it Analytic uniquely ergodic volume preserving maps on odd spheres} 


\bibitem{FK} B. Fayad and R. Krikorian,
\newblock Herman's last geometric theorem,
\newblock {Ann. Sci. \' Ec. Norm. Sup\'er.} 42, 
193--219, 2009.

\bibitem{herman-ICM} M. Herman,
\newblock 
Some open problems in dynamical systems, Proceedings of the International Congress of Mathematicians, Vol. II (Berlin, 1998),
\newblock {Doc. Math. 1998 Extra Vol. II}, 797--808, 1998.


\bibitem{katok} A.Katok, {\it Bernoulli diffeomorphisms on surfaces}, Ann. Math., 110, (1979), 529-547.

\bibitem{pesin} Y. Pesin, {\it Characteristic Lyapunov exponents, and ergodic properties of smooth dynamical systems with invariant measure. (Russian)} Dokl. Akad. Nauk SSSR {\bf 226} (1976), p. 774--777.


\bibitem{pliss} V.Pliss, {\it On a conjecture of Smale}, Diff. Uravnenjia, 8:268-282, 1972.


\bibitem{polterovich} L. Polterovich,  M. Sodin {\it A growth gap for diffeomorphisms of the interval.} J. Anal. Math. {\bf 92} (2004), p. 191--209.


\bibitem{russmann}  H. R{\"u}ssmann,
\newblock 
\"Uber die Normalform analytischer Hamiltonscher Differentialgleichungen 
in der N\"ahe einer Gleichgewichtsl\"osung,
\newblock {Math. Ann.} 169, 
55--72, 1967.

\bibitem{yoccoz} J.-C. Yoccoz, {\it Centralisateurs et conjugaison diff\'erentiable des diff\'eomorphismes du cercle,} Petits diviseurs en dimension 1, Ast\'erisque {\bf 231} (1995). 

\bibitem{yoccoz.analytic} J.-C. Yoccoz, {\it
 Analytic linearization of circle diffeomorphisms. Dynamical systems and small divisors (Cetraro, 1998)},  Lecture Notes in Math., {\bf 1784},  (2002), p. 125--173. 


\end{thebibliography}
\end{document}